\documentclass[11pt]{amsart}
\usepackage{amsfonts}
\usepackage{amssymb}
\usepackage{amsthm}
\usepackage{amsmath}
\usepackage{dsfont}
\usepackage{graphicx}
\usepackage{tikz}
\usepackage{pstricks-add}
\usepackage{verbatim}
\usepackage{color}
\usepackage{bm}
\usepackage{a4wide}
 \usetikzlibrary{arrows,automata}
\usepackage[latin1]{inputenc}
\definecolor {processblue}{cmyk}{0.96,0,0,0}
\usepackage{amsfonts,graphicx,amsmath,amssymb,hyperref,color}
\usepackage[english]{babel}
\usepackage{enumerate}

\newtheorem{theorem}{Theorem}[section]
\newtheorem{lemma}[theorem]{Lemma}
\newtheorem{proposition}[theorem]{Proposition}

 \newtheorem{main}{Theorem}

\theoremstyle{definition}
\newtheorem{definition}[theorem]{Definition}
\newtheorem{example}[theorem]{Example}

\theoremstyle{remark}
\newtheorem{remark}[theorem]{Remark}

\numberwithin{equation}{section}

\newcommand{\R}{\ensuremath{\mathbb{R}}}
\newcommand{\N}{\ensuremath{\mathbb{N}}}

\renewcommand{\c}{ {\mathbf{c}}}
\newcommand{\p}{ {\mathbf{p}}}

\newcommand{\uu} {{\mathcal{U}}}
 \newcommand{\us} {{\mathbf{U}}}
\newcommand{\ee} {{ {\mathcal{E}}}}

\newcommand{\set}[1]{\left\{#1\right\}}

\newcommand{\f}{\infty}
\newcommand{\de}{\delta}

\newcommand{\si}{\sigma}

\begin{document}

 \title{Multiple codings of self-similar sets with  overlaps}

\author{Karma Dajani}
\address[K. Dajani]{Department of Mathematics, Utrecht University,   Budapestlaan 6, P.O. Box 80.000, 3508 TA Utrecht, The Netherlands}
\email{k.dajani1@uu.nl}

\author[K.Jiang]{Kan Jiang}
\address[K. Jiang]{Department of Mathematics, Ningbo University, Ningbo, Zhejiang, People's Republic of China}
\email{kanjiangbunnik@yahoo.com}

\author[D.Kong]{Derong Kong}
\address[D. Kong]{College of Mathematics and Statistics, Chongqing University, 401331, Chongqing, P.R.China}
\email{derongkong@126.com}

\author[W.Li]{Wenxia Li}
\address[W. Li]{School of Mathematical Sciences, Shanghai Key Laboratory of PMMP, East China Normal University, Shanghai 200062,
People's Republic of China}
\email{wxli@math.ecnu.edu.cn}

\author[L.Xi]{Lifeng Xi}
\address[L. Xi]{Department of Mathematics, Ningbo University, Ningbo, Zhejiang, People's Republic of China}
\email{xilifengningbo@yahoo.com}

\date{\today}

\dedicatory{}



\begin{abstract}
In this paper we consider a general class $\mathcal E$ of self-similar sets with complete overlaps. Given a self-similar iterated function system $\Phi=(E, \{f_i\}_{i=1}^m)\in\mathcal E$ on the real line, for each point $x\in E$ we can find a sequence $(i_k)=i_1i_2\ldots\in\{1,\ldots,m\}^\mathbb N$, called a \emph{coding} of $x$, such that
\[
x=\lim_{n\to\f}f_{i_1}\circ f_{i_{2}}\circ\cdots\circ f_{i_n}(0).
\]
For $k=1,2,\ldots, \aleph_0$ or $2^{\aleph_0}$ we investigate the subset $\mathcal U_k(\Phi)$ which consists of all $x\in E$ having precisely $k$ different codings. Among several equivalent characterizations we show that $\mathcal U_1(\Phi)$ is closed if and only if $\mathcal U_{\aleph_0}(\Phi)$ is an empty set.  Furthermore, we give explicit formulae for the Hausdorff dimension of $\mathcal U_k(\Phi)$, and show that the corresponding Hausdorff measure of $\mathcal U_k(\Phi)$ is always infinite for any $k\ge 2$. Finally, we explicitly  calculate the local dimension of the self-similar measure at each point in $\mathcal U_k(\Phi)$ and ${\uu_{\aleph_0}(\Phi)}$.
\end{abstract}


\maketitle

\tableofcontents

\section{Introduction and main results}\label{sec:intro-main}

Given $\beta\in(1,2]$, for each $x\in I_\beta:=[0, 1/(\beta-1)]$ there exists a sequence $(d_i)$ of zeros and ones such that
\[
x=\sum_{i=1}^\f\frac{d_i}{\beta^i},
\]
and the sequence $(d_i)$ is called a $\beta$-\emph{expansion} of $x$.
Non-integer base expansions of reals, as a natural extension of dyadic expansions, were pioneered by R\'{e}nyi \cite{Renyi_1957} and Parry \cite{Parry_1960}. In 1990s Erd\H{o}s et al. \cite{Erdos_Horvath_Joo_1991, Erdos_Joo_1992, Erdos_Joo_Komornik_1990} discovered that for each $k=1,2,\cdots$ or $\aleph_0$ there  exist a base $\beta\in(1,2)$ and a  number  $x\in I_\beta$ such that $x$ has exactly $k$ different $\beta$-expansions. This turns out to be very different from the dyadic expansions, where each $x\in I_2$ has a unique dyadic expansion excluding {countably many} exceptions.  After the exciting {discovery} of Erd\H os and his collaborators there is a great interest in the study of non-integer base expansions. By using ergodic theorem Sidorov showed in \cite{Sidorov_2003} that for $\beta\in(1, 2)$ Lebesgue almost every  $x\in I_\beta$ has a continuum of $\beta$-expansions {(see also, \cite{Dajani_DeVries_2007})}. In the past 30 years there has been a great progress on non-integer base expansions, especially on unique $\beta$-expansions (see example, \cite{Glendinning_Sidorov_2001, DeVries_Komornik_2008, Komornik_Kong_Li_2015_1, Alcaraz-Baker-Kong-2019}). For finite $\beta$-expansions, very few is known (see \cite{Sidorov_2009, Baker_2015,Komornik-Kong-2019}).  This motivates us to study the multiple expansions (codings) in a fractal setting.
For more information on non-integer base expansions we refer the reader to the survey paper \cite{Komornik_2011} and the survey chapter \cite{deVries-Komornik-2016}.

The study of  multiple expansions also has connection to    Bernoulli convolutions.
 For $\beta\in(1,2]$ the Bernoulli convolution $\nu_{\beta}$  is defined as the  distribution of the random summation $\sum_{i=1}^{\infty}\frac{X_i}{\beta^i}$, where $X_1, X_2, \ldots$ are independent  random variables taking values  $0$ and $1$ with equal   probability. The Bernoulli convolution $\nu_\beta$ can also be viewed as the self-similar  measure  satisfying the equation
 $$\nu_{\beta}=\dfrac{1}{2}\nu_{\beta} \circ \phi_0^{-1}+\dfrac{1}{2}\nu_{\beta} \circ \phi_1^{-1},$$
where $\phi_0(x)=\frac{x}{\beta}, \phi_1(x)=\frac{x+1}{\beta}.$
Jessen and Wintner \cite{Jessen_Wintner_1935} proved that $\nu_{\beta}$ is of pure type, namely, it is either singular or absolutely continuous with respect to the Lebesgue measure. Erd\H{o}s \cite{Erdos_1939} proved that  if $\beta\in(1,2]$ is a  Pisot number, then $\nu_{\beta}$ is singular. Later, Solomyak \cite{Solomyak_1995} showed that for almost all $\beta\in(1,2]$ the Bernoulli convolution  $\nu_{\beta}$ is absolutely continuous with a density in $L^2(\mathbb{R})$.  The main difficulty in analyzing  $\nu_{\beta}$   is due to the overlapping of    the iterated function system $\{\phi_0,\phi_1\}$.  Feng and Sidorov \cite{Feng-Sidorov-2011} considered the growth rate of all $\beta$-expansions of a given point   $x\in[0, 1/(\beta-1)]$.  They showed that when $\beta$ is a Pisot number,   the growth rate is determined by the local dimension of $\nu_\beta$ at this given point $x$.
In other words, the multiple expansions have an intimate connection with  the local dimension of the   Bernoulli convolution.

Another reason for us to  investigate the multiple expansions comes from a seemingly unrelated development in arithmetic progressions.
An \emph{arithmetic progression} in $\mathbb{R}$ is of the form
\begin{equation*}
P=\{a,a+\delta ,a+2\delta ,a+3\delta ,\cdots ,a+(k-1)\delta \}
\end{equation*}%
for some $a\in \mathbb{R}$, $\delta \in \mathbb{R}^{+}$ and $k\in \mathbb{N}%
^{+}$. In this case   $k$ is called the length of the arithmetic progression $P$. To find an arithmetic progression in a  given set is a crucial problem in combinatorial number theory. Erd\H{o}s and Tur\'{a}n \cite{ErdosPaul}
conjectured that a subset of natural numbers with positive density necessarily
contains arbitrarily long arithmetic progressions. This conjecture was proved by   Szemer\'{e}di \cite{Szemeredi1,Szemeredi2} with a delicate combinatorial analysis.
Furstenberg \cite{Furstenberg} later proved that this conjecture is equivalent to the multiple recurrence theorem in ergodic theory.
 Motivated by this it is natural to consider arithmetic progressions in a fractal setting.   For some recent progress  in this direction, see \cite{Fraser-Yu-2018,Bro-Fis-Sim-2017,Jon,Laba2016} and references therein.
Recently, the second author and his collaborators  \cite{Jiang-Pei-Xi-2019} made use of   multiple $\beta$-expansions  to study the arithmetic progressions.
This constructed a `bridge' between arithmetic progressions and  multiple expansions.

In this paper we   consider multiple expansions for self-similar sets with overlaps.
For $1\le i\le m$ let $f_i(\cdot)$ be a \emph{similitude} on $\mathbb{R}$ defined by
\[
f_i(x)=r_i x+b_i,
\]
where $r_i\in(0,1)$ and $b_i\in\mathbb{R}$. Then there exists a unique non-empty compact set $E{\subset\R}$ satisfying (cf.~\cite{Hutchinson_1981})
\[
E=\bigcup_{i=1}^m f_i(E).
\]
In this case,   the couple $\Phi=(E, \{f_i\}_{i=1}^m)$ is called a \emph{self-similar iterated function system} (SIFS), and the compact set $E$ is called a \emph{self-similar set} generated by $\{f_i\}_{i=1}^m$. 

Now we introduce a   class $\mathcal{E}$ of SIFS $\Phi=(E, \{f_i\}_{i=1}^m)$ on $\R$ which will be  our object throughout  the paper.  Denote by $I=[a,b]$ the convex hull of the self-similar set $E$. We say that $\Phi\in\ee$ if it satisfies the following conditions {(A)--(D)}.
\begin{enumerate}
  \item[(A)]   $a=f_1(a)<f_2(a)<\cdots<f_m(a)<f_m(b)=b$.

  \item [(B)] $f_i(I)\cap f_{i+2}(I)=\emptyset$ for any $1\le i\le m-2$.

\item[(C)] There exist $i, j\in\{1,\cdots,m-1\}$ such that
\[f_i(I)\cap f_{i+1}(I){\ne}\emptyset\quad\textrm{ and}\quad f_j(I)\cap f_{j+1}(I){=}\emptyset.\]

  \item[(D)] If $f_i(I)\cap f_{i+1}(I)\ne \emptyset$, then there exist positive integers $u_i$ and $v_i$ such that
      \[
       f_i(I)\cap f_{i+1}(I)=f_{i m^{u_i}}(I)=f_{(i+1)1^{v_i}}(I),
      \]
 where $f_{i_1\cdots i_k}(\cdot):=f_{i_1}\circ\cdots\circ f_{i_k}(\cdot)$ denotes the compositions of the maps $f_{i_1}, \ldots, f_{i_k}$.
\end{enumerate}
The intervals $f_i(I), i=1,\cdots,m$ are called  the \emph{basic intervals} of $\Phi=(E, \{f_i\}_{i=1}^m)$.
 Then by Conditions (A)--(D) it follows that $m\ge 3$, and the basic intervals are located from left to right in the following way {(see, e.g., Figure \ref{fig:1})}.  The most left one is $f_1(I)$, and   the second one is $f_2(I)$, and the most right one is $f_m(I)$.  Furthermore, there exist  two neighboring basic intervals having a non-empty intersection, and also exist two neighboring basic intervals   having an empty intersection. But any three basic intervals must have an empty intersection. By Condition (D) it follows that {any} basic interval cannot be {included} in another {basic} interval, and the intersection of basic intervals cannot be a singleton.

 \begin{figure}[h!]
\begin{center}
\begin{tikzpicture}[
    scale=8,
    axis/.style={very thick, ->},
    important line/.style={thick},
    dashed line/.style={dashed, thin},
    pile/.style={thick, ->, >=stealth', shorten <=2pt, shorten
    >=2pt},
    every node/.style={color=black}
    ]


    \draw[important line]  (0,0)--(1,0);
      \node[] at (0,0.05){$0$};

        \node[] at (1, 0.05){$1$};
        \node[] at (1/2,0.05){$I$};
            \draw[important line] (0,-0.3)--(1/4, -0.3);

    \draw[important line] (3/16,-0.29)--(7/16, -0.29);

        \draw[important line] (33/64, -0.3)--(49/64, -0.3);
                \draw[important line] (3/4, -0.29)--(1, -0.29);

                 \node[] at (1/8, -0.33){$f_1(I)$};
                   \node[] at (5/16, -0.26){$f_2(I)$};
                    \node[] at (41/64,-0.33){$f_3(I)$};
                     \node[] at (7/8,-0.26){$f_4(I)$};
\end{tikzpicture}
\end{center}
\caption{The  basic intervals $f_1(I), f_2(I), f_3(I)$ and $f_4(I)$ with $I=[0, 1]$, and the maps $f_i, 1\le i\le 4$ are defined as in (\ref{eq:maps-1}) with $\beta=4, M=1, N=2$. Then $f_1(I)\cap f_2(I)=f_{14}(I)=f_{21}(I)$ and $f_{3}(I)\cap f_4(I)=f_{344}(I)=f_{411}(I)$. So the SIFS $\Phi=(E, \set{f_i}_{i=1}^4)\in\ee$. Then $\dim_H\uu_k(\Phi)=\dim_H\uu_1(\Phi)\approx 0.0943436$ and $\dim_H\uu_{2^{\aleph_0}}(\Phi)=\dim_H E\approx 0.934164$. See Example \ref{ex:1} for more explanations. }\label{fig:1}
\end{figure}

 Here we mention that Condition (D) is always associated with {the \emph{complete overlap} condition} (cf.~\cite{Hochman-2014}).
 This  class of complete overlapping  SIFSs has been studied by many people from different aspects. Ngai and Wang \cite{Ngai_Wang_2001} calculated the Hausdorff dimension of the overlapping self-similar sets. Rao and Wen \cite{Rao-Wen-1998} considered the topology of a special class of overlapping self-similar sets $E_\lambda$ generated by
$
 f_1(x)=\frac{x}{3},$ $f_2(x)=\frac{x+\lambda}{3}$ and $f_3(x)=\frac{x+2}{3},
$
 where $\lambda\in[0, 1]$. In particular, they determined for which rational number $\lambda$ the self-similar set $E_\lambda$ contains interior points (Kenyon \cite{Keyon_1997} proved similar result).
 The generating IFSs for $E_\lambda$ was recently investigated by Dajani et al.~\cite{Dajani-Kong-Yao-2019}. Guo et al.~\cite{Guo-Li-Wang-Xi-2012} considered the Lipschitz equivalence for a class of self-similar sets with complete overlaps.

Let $\Phi=(E, \{f_i\}_{i=1}^m)\in \mathcal{E}$. Then for any $x\in E$ there exists a sequence $(d_i)=d_1d_2\cdots\in\{1,2,\cdots,m\}^{\N}$ such that  (cf.~\cite{Falconer_1990})
\begin{equation}
  \label{eq:k1}
  x=\lim_{n\to\infty}f_{d_1\cdots d_n}(0)=:\pi((d_i)).
\end{equation}
The sequence $(d_i)$ is called a \emph{coding} of $x$ with respect to the \emph{alphabet} $\set{1,\ldots, m}$.  We point out that a point $x\in E$ may have multiple codings by {Conditions (C) and (D)}.
For $k=1,2,\cdots, \aleph_0$ or $2^{\aleph_0}$  we set
\[
\mathcal{U}_k(\Phi):=\{x\in E: x~\textrm{has exactly}~ k~\textrm{different codings}  \}.
\]
  Recently, the authors \cite{Dajani_Jiang_Kong_Li-2015} considered a special candidate $\Phi_0=(E, \{f_i\}_{i=1}^3)$ of $\ee$ which was first introduced by Ngai and Wang \cite{Ngai_Wang_2001}, where
\[
 f_1(x)=\frac{ x}{\beta}, \quad f_2(x)=\frac{ x+1}{\beta},\quad f_3(x)=\frac{ x+\beta}{\beta}
\]
with  $\beta>(3+\sqrt{5})/2$. Based on the characterization of $\uu_1(\Phi_0)$ given by Zou et al.~\cite{Zou_Lu_Li_2012} they showed that the Hausdorff dimensions of $\mathcal{U}_k(\Phi_0)$ are the same for all {integers} $k\ge 1$.

Inspired by the work of  \cite{Dajani_Jiang_Kong_Li-2015} we  investigate the multiple codings of the SIFSs in $\mathcal{E}$ from different aspects. (I) We classify the collection $\ee$ via the multiple codings set $\uu_k(\Phi)$, see Theorem \ref{th:1}; (II) We calculate in Theorem \ref{th:measures}  the Hausdorff dimension   of $\uu_k(\Phi)$, and show that $\uu_k(\Phi)$ has infinite Hausdorff measure for any $k\ge 2$ assuming  $\uu_k(\Phi)\ne \emptyset$; (III)   We determine in Theorem \ref{th:local dimensions} the local dimension of the self-similar measure at points in $\uu_k(\Phi)$ and $\uu_{\aleph_0}(\Phi)$.

\subsection{Classifications of $\mathcal E$}
Our first result focuses on the classification of    $\ee$. For a set $X$ we denote by $|X|$ its cardinality.

\begin{main}
  \label{th:1}
  Let $\Phi=(E, \{f_i\}_{i=1}^m)\in \mathcal{E}$. Denote by $I=[a, b]$ the convex hull of $E$.
  \begin{itemize}
  \item[{\rm(A)}.] The following statements are equivalent.
  \begin{enumerate}[{\rm(i)}]
  \item  $f_1(I)\cap f_2(I)\ne\emptyset$ or $f_{m-1}(I)\cap f_m(I)\ne\emptyset$.
  \item
  $
  \dim_H \mathcal{U}_k(\Phi)=\dim_H\mathcal{U}_1(\Phi)
$ for all {  integers} $k\ge 1$.
 \item   $ f_1(b)\in\mathcal{U}_{\aleph_0}(\Phi)$ or $ f_m(a)\in\mathcal{U}_{\aleph_0}(\Phi)$.
\item  $|\mathcal{U}_{\aleph_0}(\Phi)|=\aleph_0$.

\item   $\mathcal{U}_{1}(\Phi)$ is  not closed.
\end{enumerate}

\item[{\rm(B)}.] The following statements are also equivalent.
  \begin{enumerate}[{\rm(i)}]
   \item  $f_1(I)\cap f_2(I)=f_{m-1}(I)\cap f_m(I)=\emptyset$.
   \item  $ \dim_H\mathcal{U}_k(\Phi)=\dim_H\mathcal{U}_1(\Phi)$ if $k=2^\ell$ {  with $\ell\in \mathbb{N}\cup\set{0}$}, {or} $\mathcal{U}_k(\Phi)=\emptyset$ otherwise.
 \item   $ f_1(b)\notin\mathcal{U}_{\aleph_0}(\Phi)$ and $ f_m(a)\notin\mathcal{U}_{\aleph_0}(\Phi)$.
 \item   $\mathcal{U}_{\aleph_0}(\Phi)=\emptyset$.
     \item   $\mathcal{U}_{1}(\Phi)$ is closed.
\end{enumerate}
\end{itemize}
\end{main}

\begin{remark}\mbox{}

\begin{itemize}
{\item In the next  result Theorem \ref{th:measures} we will show that for any integer $k\ge 1$ with $\uu_k(\Phi)\ne\emptyset$ we must have $\dim_H\uu_k(\Phi)>0$.}

\item Theorem \ref{th:1} classifies the collection $\mathcal E$ and gives  several  dichotomies. For example,  for any $\Phi\in \mathcal{E}$, either
${|\mathcal{U}_{\aleph_0}(\Phi)|=\aleph_0}$
 or  $\mathcal{U}_{\aleph_0}(\Phi)=\emptyset.$

 \item Although the closeness of $\uu_1(\Phi)$ can be used to classify  $\ee$.  This does not mean the same holds for  $\uu_k(\Phi)$ with $k\ge 2$. In fact, we show in Propositions \ref{prop:notclose} and \ref{prop:close-notclose} that for any $k\ge 2$ with $\uu_k(\Phi)\ne\emptyset$, the set $\uu_k(\Phi)$ is not closed.

\end{itemize}
\end{remark}

 \subsection{Hausdorff dimension and Hausdorff measure of $\mathcal U_k(E)$}
In general, without the complete overlap condition $(D)$ it is hard to calculate the Hausdorff  dimensions of {$E$ and $\mathcal{U}_k(\Phi)$}. In our class $\mathcal{E}$ we are able to explicitly determine the Hausdorff dimensions  of the attractor $E$ and the set $\uu_k(\Phi)$.  Furthermore,   we show that  for $k\ge 2$ with $\uu_k(\Phi)\ne\emptyset$ the corresponding Hausdorff measure of $\uu_k(\Phi)$ is always infinite.

\begin{definition}\label{def:overlap-vector}
Given $\Phi=(E, \set{f_i}_{i=1}^m)\in\ee$,   the \emph{overlapping   vectors} $\mathbf u=(u_1,\ldots, u_m), \mathbf v=(v_1,\ldots, v_m)$ of $\Phi$  are defined  by
\[
\left\{\begin{array}{lll}
u_i=u, ~ v_i=v&\quad\textrm{if}\quad& f_i(I)\cap f_{i+1}(I)=f_{im^u}(I)=f_{(i+1)1^v}(I),\\
u_i=v_i=\f&\quad\textrm{if}\quad& f_i(I)\cap f_{i+1}(I)=\emptyset.
\end{array}\right.
\]
\end{definition}
Then by Definition \ref{def:overlap-vector} $u_m$ and $v_m$ are  always equal to $\f$. Observe by Theorem \ref{th:1} that $\uu_{\aleph_0}(\Phi)$ is either a countable set or an empty set. So it suffices to consider the Hausdroff dimension and Hausdorff measure of $\uu_k(\Phi)$ for $k=1,2,\ldots$ or $k=2^{\aleph_0}$.

\begin{main}\label{th:measures}
Let $\Phi=(E, \set{f_i}_{i=1}^m)\in\ee$ with  $f_i(x)=r_i x+b_i$ for $1\le i\le m$, and {let $\mathbf u=(u_1,\ldots, u_m), \mathbf v=(v_1,\ldots, v_m)$ be the overlapping vectors} defined as in Definition \ref{def:overlap-vector}.

\begin{enumerate}[{\rm(i)}]
\item The Hausdorff dimensions of $\uu_{2^{\aleph_0}}$ and $E$ are given by
\[
\dim_H \uu_{2^{\aleph_0}}(\Phi)=\dim_H E=t,
\]
where $t\in(0, 1)$ is the root of
\[ \sum_{i=1}^m r_i^t (1-r_m^{u_i t})=1.\]
Furthermore, the corresponding Hausdorff measures are positive and finite, i.e.,
\[
\mathcal H^{t}(\uu_{2^{\aleph_0}}(\Phi))=\mathcal H^{t}(E)\in(0, \f).
\]

\item For any $k\ge 1$ satisfying $\uu_k(\Phi)\ne\emptyset$, the Hausdorff dimension of $\uu_k(\Phi)$ is given by
\[
 \dim_H\uu_k(\Phi)=\dim_H\uu_1(\Phi)=s,
\]
where $s\in(0, 1)$ is the root of
\[
 \sum_{i=1}^m r_i^s\left(1-\frac{r_m^{u_i s}(2-r_1^{v_{m-1}s}-r_m^{u_1 s})}{1-r_1^{v_{m-1}s}r_m^{u_1 s}}\right)=1.
\]
Furthermore,  the corresponding Hausdorff measure of   $\uu_k(\Phi)$ admits
\[
 \mathcal H^s(\uu_1(\Phi))\in(0,\f)
\]
 and
 \[
  \mathcal H^s(\uu_k(\Phi))=\f\quad\textrm{for any }k\ge 2\textrm{ satisfying }\uu_k(\Phi)\ne\emptyset.
 \]
\end{enumerate}
\end{main}

\begin{remark}\mbox{}

\begin{itemize}
\item {By Propositions \ref{prop:measure-u1} and \ref{prop:measure-E} the sets $\uu_1(\Phi)$ and $E$} are both identical to strongly connected graph-directed sets satisfying the \emph{open set condition} (OSC).  So by \cite{Mauldin_Williams_1988} the results obtained in Theorem \ref{th:measures} also hold for packing dimension and corresponding packing measure.

\item After the paper was finished we notice that Theorem \ref{th:measures} (i) for the Hausdorff dimension of $E$ was also studied by Deng et al.~\cite[Theorem 8]{Deng-Harding-Hu-2009}. However, our method is different from theirs, where they determined the dimension of   $E$ by considering  an infinite iterated function system satisfying the OSC.

\item The Hausdorff dimension formula for $\uu_k(\Phi)$ described in Theorem \ref{th:measures} (ii) is in a compact form.  If $\Phi=(E, \set{f_i}_{i=1}^m)\in\ee$ with $f_1(I)\cap f_2(I)=f_{m-1}(I)\cap f_m(I)=\emptyset$, then $u_1=v_{m-1}=\f$. By Theorem \ref{th:measures} (ii) the Hausdorff dimension of $\uu_k(\Phi)$ can be simplified as $\dim_H\uu_k(\Phi)=s$, where $s\in(0, 1)$ is the root of
\[
1=\sum_{i=1}^m r_i^s(1-2 r_m^{u_i s}).
\]

  Similarly, if  $\Phi=(E, \set{f_i}_{i=1}^m)\in\ee$ with $f_1(I)\cap f_2(I)\ne\emptyset$ and $f_{m-1}(I)\cap f_m(I)=\emptyset$, then $u_1\in\N$ and $v_{m-1}=\f$. Again by Theorem \ref{th:measures} (ii) it follows that $\dim_H\uu_k(\Phi)=s$ satisfies
\[
1=\sum_{i=1}^m r_i^s(1-2 r_m^{u_i s}+r_m^{(u_i+u_1)s}).
\]

 Also, if   $\Phi=(E, \set{f_i}_{i=1}^m)\in\ee$ with $f_1(I)\cap f_2(I)=\emptyset$ and $f_{m-1}(I)\cap f_m(I)\ne \emptyset$, then $u_1=\f$ and $v_{m-1}\in\N$. By Theorem \ref{th:measures} (ii) the Hausdorff dimension of $\uu_k(\Phi)$ is given by $\dim_H\uu_k(\Phi)=s$, where $s\in(0, 1)$ is the root of
\[
1=\sum_{i=1}^m r_i^s(1-2 r_m^{u_i s}+r_m^{u_i s}r_1^{v_{m-1} s}).
\]
\end{itemize}
\end{remark}

 \subsection{Local dimension of self-similar measure in $\uu_k(\Phi)$ and $\uu_{\aleph_0}(\Phi)$}
 Given a probability vector $\p=(p_1,p_2,\ldots, p_m)$ with each $p_i>0$, let $\mu_\p$ be the self-similar measure defined on $\Phi=(E, \set{f_i}_{i=1}^m)\in\ee$ (cf.~\cite{Hutchinson_1981}). Then
 \[
 \mu_\p=\sum_{i=1}^m p_i\mu_\p\circ f_i^{-1}.
 \]
 The measure $\mu_\p$ can also be deduced from the projection of Bernoulli measure on the symbolic space $\set{1,\ldots, m}^\N$. More precisely, let $\nu_\p$ be the Bernoulli measure on $\set{1,\ldots,m}^\N$ defined by
 \[\nu_\p([i])=p_i\quad \textrm{for}\quad i=1,\ldots,m,\]
  where $[i]:=\set{(j_\ell)\in\set{1,\ldots, m}^\N: j_1=i}$ is a cylinder set. Then
 \[\mu_\p=\nu_\p\circ\pi^{-1},\]
  where $\pi$ is the projection map defined in (\ref{eq:k1}).

  Given $x\in E$, we define the \emph{lower and upper local dimensions} of $\mu_\p$ at $x$ by
  \[
  \underline{\dim}_{loc}\mu_\p(x):=\liminf_{r\to 0}\frac{\log \mu_\p(B(x,r))}{\log r},\quad \overline{\dim}_{loc}\mu_\p(x):=\limsup_{r\to 0}\frac{\log \mu_\p(B(x,r))}{\log r},
  \]
  where $B(x,r)=(x-r,x+r)$ is the open ball in $\R$ with center at $x$ and radius $r$. If the lower and upper local dimensions coincide, then the common value, denoted by $\dim_{loc}\mu_\p(x)$, is called the \emph{local dimension} of $\mu_\p$ at $x$.

  To determine the local dimension of $\mu_\p$ is a central topic of multifractal analysis. Recently, Ngai and Xie \cite{Ngai-Xie-2018} calculated the Hausdorff dimension of the self-similar measure $\mu_\p$, which provides a typical value of   $\dim_H\mu_\p(x)$.
For more result on the multifractal analysis  of $\mu_\p$ we refer to the papers of Lau and Ngai \cite{Lau-Ngai-1999}, Feng and Lau \cite{Feng-Lau-2009}, and references therein.

As a compensation of \cite{Ngai-Xie-2018} we explicitly calculate the  local dimension of $\mu_\p$ at points in $\uu_k(\Phi)$ and $\uu_{\aleph_0}(\Phi)$.
For $n\in\N$ let $\set{1,\ldots, m}^n$ be the set of all length $n$ words over the alphabet $\set{1,\ldots, m}$. Denote by $\set{1,\ldots, m}^*=\bigcup_{n=0}^\f\set{1,\ldots,m}^n$ the set of all words, where for $n=0$ we set $\set{1,\ldots, m}^0=\set{\epsilon}$ with $\epsilon$ the empty word.
 \begin{main}
 \label{th:local dimensions}
 Let $\Phi=(E, \set{f_i}_{i=1}^m)\in\ee$ with the convex hull $conv(E)=[a, b]$, and {let  $\mathbf u=(u_1,\ldots, u_{m}), \mathbf v=(v_1,\ldots, v_{m})$ be the overlapping vector} defined as in Definition \ref{def:overlap-vector}.
 \begin{enumerate}[{\rm(i)}]
 \item If $x\in\uu_k(\Phi)$, then there exist a word $\mathbf i\in\set{1,\ldots, m}^*$ and a unique $y\in\uu_1(\Phi)$  such that $x=f_{\mathbf i}(y)$. Furthermore,
 \begin{align*}
 \underline{\dim}_{loc}\mu_\p(x)&=\underline{\dim}_{loc}\mu_\p(y)=\liminf_{n\to\f}\frac{\sum_{k=1}^n\log p_{j_k}}{\sum_{k=1}^n\log r_{j_k}},\\
  \overline{\dim}_{loc}\mu_\p(x)&=\overline{\dim}_{loc}\mu_\p(y)=\limsup_{n\to\f}\frac{\sum_{k=1}^n\log p_{j_k}}{\sum_{k=1}^n\log r_{j_k}},
 \end{align*}
 where $j_1j_2\ldots$ is the unique coding of $y$.

 \item If $x\in\uu_{\aleph_0}(\Phi)$, then either $x=f_{\mathbf i}(f_1(b))$ for some $\mathbf i\in\set{1,\ldots, m}^*$ and then
 \[
 \dim_{loc}\mu_\p(x)=\min\set{\frac{\log p_m}{\log r_m}, \frac{\log p_2+(v_1-1)\log p_1}{u_1\log r_m}},
 \]
 or $x=f_{\mathbf i}(f_m(a))$ for some $\mathbf i\in\set{1,\ldots, m}^*$ and thus
 \[
 \dim_{loc}\mu_\p(x)=\min\set{\frac{\log p_1}{\log r_1}, \frac{\log p_{m-1}+(u_{m-1}-1)\log p_m}{v_{m-1}\log r_1}}.
 \]
 \end{enumerate}
 \end{main}

\subsection{An example}
 At the end of this section we give an example to illustrate the main results Theorems \ref{th:1}--\ref{th:local dimensions}.

 \begin{example}\label{ex:1}
 Given two  integers $M, N\ge 1$, let $\beta\in(1,4]$ satisfy
 \begin{equation}\label{eq:ex-1}
 \frac{4}{\beta}<1+\frac{1}{\beta^M}+\frac{1}{\beta^N}.
 \end{equation}
 Define the maps
 \begin{equation}\label{eq:maps-1}
 \begin{split}
 f_1(x)&=\frac{x}{\beta},\hspace{3.2cm}  f_2(x)=\frac{x}{\beta}+\frac{1}{\beta}-\frac{1}{\beta^{M+1}},\\
 f_3(x)&=\frac{x}{\beta}+1-\frac{2}{\beta}+\frac{1}{\beta^{N+1}},\quad f_4(x)=\frac{x}{\beta}+1-\frac{1}{\beta}.
 \end{split}
 \end{equation}
 Then there exists a unique non-empty compact set $E$ satisfying $E=\bigcup_{i=1}^4 f_i(E)$. It is easy to check that the convex hull of $E$ is the unit interval $I=[0, 1]$. The basic intervals $f_i(I)$ with $1\le i\le 4$ are plotted as in Figure \ref{fig:1}.
 The inequality in (\ref{eq:ex-1}) guarantees that $f_2(I)\cap f_3(I)=\emptyset$. Furthermore, one can verify that
 \begin{align*}
 f_1(I)\cap f_2(I)&=f_{14^M}(I)=f_{21^M}(I),\\
 f_3(I)\cap f_4(I)&=f_{34^N}(I)=f_{41^N}(I).
 \end{align*}
 So, $\Phi=(E, \set{f_i}_{i=1}^4)\in\ee$ for any integers $M, N\ge 1$.

First, by Theorems \ref{th:1} and \ref{th:measures} it follows that for any $k\in\N$,
\[
\dim_H\uu_k(\Phi)=\dim_H\uu_1(\Phi)=s,
\]
where $s\in(0,1)$ satisfies
\[
\beta^s+2\beta^{N s}+2 \beta^{M s}=4+\beta^{(N-M)s}+\beta^{(M-N)s}.
\]
Moreover, the corresponding Hausdorff measures satisfy
\[
\mathcal H^s(\uu_1(\Phi))\in(0, \f),\quad\textrm{and}\quad \mathcal H^s(\uu_k(\Phi))=\f\quad\forall ~ k\ge 2.
\]

Second, by Proposition \ref{cor:countable} it follows that $\uu_{\aleph_0}(\Phi)$ is a countable set consisting of all points with a coding ending with either $14^\f$ or $4 1^\f$. Furthermore, by Theorem \ref{th:measures} it follows that the Hausdorff dimensions of $\uu_{2^{\aleph_0}}(\Phi)$ and $E$ are given by
\[
\dim_H\uu_{2^{\aleph_0}}(\Phi)=\dim_H E=t,
\]
where $t\in(0, 1)$ satisfies
\[
\beta^t+\beta^{-Mt}+\beta^{-Nt}=4.
\]
And the corresponding Hausdorff measures admit
\[
\mathcal H^t(\uu_{2^{\aleph_0}}(\Phi))=\mathcal H^t(E)\in(0, \f).
\]

 Third, for a given probability vector $\p=(p_1, p_2, p_3, p_4)$ with each $p_i>0$ we define the corresponding self-similar measure $\mu_\p$ by
 \[
 \mu_\p=\sum_{i=1}^4 p_i\mu_\p\circ f_i^{-1}.
 \]
 Then by Theorem \ref{th:local dimensions} it follows that for any $x\in\uu_k(\Phi)$ there exits a unique $y\in\uu_1(\Phi)$ such that
  \begin{align*}
 \underline{\dim}_{loc}\mu_\p(x)&=\underline{\dim}_{loc}\mu_\p(y)=\frac{-1}{\log \beta}\limsup_{n\to\f}\frac{\sum_{k=1}^n\log p_{j_k}}{n},\\
  \overline{\dim}_{loc}\mu_\p(x)&=\overline{\dim}_{loc}\mu_\p(y)=\frac{-1}{\log \beta}\liminf_{n\to\f}\frac{\sum_{k=1}^n\log p_{j_k}}{n},
 \end{align*}
 where $j_1j_2\ldots$ is the unique coding of $y$. Moreover, for any $x\in\uu_{\aleph_0}(\Phi)$ we have the following two cases:
 \[x=f_{\mathbf i}\left(\frac{1}{\beta}\right)\quad \textrm{or}\quad x=f_{\mathbf i}\left(1-\frac{1}{\beta}\right)\]
  for some $\mathbf i\in\set{1,2,3,4}^*$. If $x=f_{\mathbf i}(\frac{1}{\beta})$, then
 \[
 \dim_{loc}\mu_\p(x)=\min\set{\frac{\log p_4}{-\log \beta}, \frac{\log p_2+(M-1)\log p_1}{-M\log \beta}}.
 \]
 If $x=f_{\mathbf i}(1-\frac{1}{\beta})$, then
 \[
 \dim_{loc}\mu_\p(x)=\min\set{\frac{\log p_1}{-\log \beta}, \frac{\log p_3+(N-1)\log p_4}{-N\log \beta}}.
 \]
 In particular, if $p_i=\frac{1}{4}$ for all $1\le i\le 4$, then
 \[
 \dim_{loc}\mu_\p(x)=\frac{\log 4}{\log \beta}\quad\textrm{for all }x\in\uu_{\aleph_0}(\Phi).
 \]
 \end{example}

 \medskip

The rest of the paper is organized as follows. The proof of Theorem \ref{th:1} is given in Sections \ref{sec:finite codings}--\ref{sec:closeness}. First in   Section \ref{sec:finite codings} we characterize the set of points in $E$ with finitely many codings, and prove the equivalence (i) $\Leftrightarrow$ (ii) in Theorem \ref{th:1} (A) and (B). Second in Section \ref{sec:countable codings} we describe the set of points in $E$ with countably   many codings, and establish the equivalence (i) $\Leftrightarrow$ (iii) $\Leftrightarrow$ (iv). Finally in Section \ref{sec:closeness} we investigate the topology of $\uu_k(\Phi)$, and deduce the equivalence (i) $\Leftrightarrow$ (v). In Section \ref{sec:geometrical structure} we show that the univoque set $\uu_1(\Phi)$ and the self-similar set $E$ are both identical to the strongly connected graph-directed sets satisfying the OSC.  Based on this we prove Theorem \ref{th:measures} in Section \ref{sec:hausdorff measure-uk}   for the Hausdorff dimensions and Hausdorff measures of $E$ and $\uu_k(\Phi)$. In particular, we show  that the corresponding Hausdorff measure of $\uu_k(\Phi)$ is always infinite for any $k\ge 2$ satisfying $\uu_k(\Phi)\ne\emptyset$. In Section \ref{sec:local dimension} we prove Theorem \ref{th:local dimensions} for the local dimension of the self-similar measure $\mu_\p$ at points in $\uu_k(\Phi)$ and $\uu_{\aleph_0}(\Phi)$. Finally, in Section \ref{sec:final remarks} we pose some remarks on a possible extension of our  class $\ee$.

\section{Finitely many codings}\label{sec:finite codings}
Let $\Phi=(E, \{f_i\}_{i=1}^m)\in\mathcal{E}$ and $k\in\N$. In this section we will consider the set $\mathcal{U}_k(\Phi)$ consisting of all $x\in E$ having precisely $k$ different codings with respect to $\{f_i\}_{i=1}^m$, and prove the equivalence  $(i)\Leftrightarrow(ii)$ in Theorems \ref{th:1} (A) and (B), respectively.

 Let $\set{1,\ldots, m}^\N$ be the set of sequences $(d_i)$ with each $d_i\in\set{1,\ldots, m}$, and let $\set{1,\ldots, m}^*$ be the set of all finite words over the alphabet $\set{1,\ldots, m}$.  For any two words $\mathbf c=c_1\ldots c_m, \mathbf d=d_1\ldots d_n\in\set{1,\ldots, m}^*$ we denote by $\mathbf c\mathbf d=c_1\ldots c_m d_1\ldots d_n$ their concatenation. Moreover, for $k\in\N$ we write for $\mathbf c^k=\mathbf c\mathbf c\cdots\mathbf c$
 the $k$ times concatenation of $\mathbf c$ with itself, and denote by $\mathbf c^\f$ the infinite concatenation of $\mathbf c$ with itself. Recall that $\set{f_i}_{i=1}^m$ is the collection of contractive similitudes.
For a word $\mathbf d=d_1d_2\ldots d_n\in\set{1,\ldots, m}^*$ we denote by $f_{\mathbf d}=f_{d_1}\circ f_{d_{2}}\circ\cdots \circ f_{d_n}$ the compositions of the maps $f_{d_1}, \ldots, f_{d_n}$. In particular, for the empty word $\epsilon$ we set $f_{\epsilon}$ the identity map.

For a set $F\subset \R$ we denote by $conv(F)$ the \emph{convex hull} of $F$.
\begin{lemma}
  \label{lem:21}
  Let $(E,\{f_i\}_{i=1}^m)\in\mathcal{E}$ with $I=conv(E)$. If
$
 f_{im^{u}}(I)=f_{(i+1)1^{v}}(I)
$
  for some $1\le i\le m-1$ and $u, v\in \N$,
  then
$f_{im^{u}}(\cdot)=f_{(i+1)1^{v}}(\cdot).$
\end{lemma}
\begin{proof}
  Note that for any $x\in\R$ we can write
  \begin{equation}\label{eq:k2}
  f_{im^{u}}(x)=r x+t,\quad f_{(i+1)1^{v}}(x)=r' x+t',
  \end{equation}
for some $r, r'\in(0,1)$ and $t, t'\in\R$. Suppose that $I=[a, b]$. Then by using $f_{im^{u}}(I)=f_{(i+1)1^{v}}(I)$ it follows that
\begin{align*}
r a+t&=f_{im^{u}}(a)=f_{(i+1)1^{v}}(a)=r' a+t',\\
 rb+t&=f_{im^{u}}(b)=f_{(i+1)1^{v}}(b)=r' b+t'.
\end{align*}
This implies $r=r'$ and $t=t'$. By (\ref{eq:k2}) we have $f_{im^{u}}(\cdot)=f_{(i+1)1^{v}}(\cdot)$.
\end{proof}

\begin{lemma}
  \label{lem:key lem}
  Let $(E, \{f_i\}_{i=1}^m)\in\mathcal{E}$ with $I=conv(E)$. If
\[
  x\in f_i(I)\cap f_{i+1}(I)=f_{im^{u}}(I)=f_{(i+1)1^{v}}(I)
\]
for some $i\in\{1,\cdots,m-1\}$ and $u, v\in\N$, then all codings of $x$ either begin with $im^{u-1}$  or begin with $(i+1)1^{v-1}$.
\end{lemma}
\begin{proof}
Let $(d_i)$ be a coding of $x$ with respect to $\{f_i\}_{i=1}^m$. Note that $\pi(d_1d_2\ldots)=x\in f_i(I)\cap f_{i+1}(I)$ and that any three basic intervals have an empty intersection. Then
\[d_1=i\quad\textrm{ or} \quad d_1=i+1.\]
{If $d_1=i$ with $u=1$  or $d_1=i+1$ with $v=1$}, then we are done. So, we will finish the proof by considering   the following two cases.

Case (I). $d_1=i$ and $u>1$. Note that $x=\pi(i d_2d_3\cdots)\in f_{im^{u}}(I)$. Then
\begin{equation}
  \label{eq:a1}
 \pi(d_2d_3\cdots)\in f_{m^{u}}(I),
\end{equation}
{ and we claim that $d_2=m$.}

  Suppose on the contrary that $d_2\ne m$. Then in view of the location of these basic  intervals we have $d_2=m-1$. So, by (\ref{eq:a1}) and Condition (D) it follows
  that
  \[
  \pi(d_2d_3\cdots)\in f_{m^{u}}(I)\cap\big(f_{m-1}(I)\cap f_m(I)\big)\subseteq f_{m^{u}}(I)\cap f_{m1}(I),
  \]
  leading to a contradiction with $f_1(I)\cap f_m(I)=\emptyset$.
  Therefore, $d_2=m$. Iterating the above procedure  it follows that $d_2\cdots d_{u}=m^{u-1}$.

  Case (II). $d_1=i+1$ and $v>1$. Note that $x=\pi((i+1)d_2d_3\cdots)\in f_{(i+1)1^{v}}(I)$. Then
  \begin{equation}
    \label{eq:a2}
   \pi(d_2d_3\cdots)\in f_{1^{v}}(I),
  \end{equation}
  { and we will prove that $d_2=1$.}
If $d_2\ne 1$, then $d_2=2$. So, by (\ref{eq:a2}) and Condition (D) it follows that
  \[
  \pi(d_2d_3\cdots)\in f_{1^{v}}(I)\cap \big(f_1(I)\cap f_2(I)\big)\subseteq f_{1^{v}}(I)\cap f_{1m}(I),
  \]
  leading to a contradiction with $f_1(I)\cap f_m(I)=\emptyset$.
 Hence $d_2=1$. By iteration we conclude that $d_2\cdots d_{v}=1^{v-1}$.
\end{proof}
Now we turn to show that the Hausdorff dimensions of $\uu_k(\Phi)$ and $\uu_1(\Phi)$ coincide  for any $k\ge 2$ satisfying $\uu_k(\Phi)\ne\emptyset$. Later on in Section \ref{sec:hausdorff measure-uk} we will explicitly calculate the Hausdorff dimension of $\uu_1(\Phi)$. In particular, we have $\dim_H\uu_1(\Phi)>0$ for any $\Phi\in\ee$.
The upper bound of $\dim_H\mathcal{U}_k(\Phi)$ follows directly.
 \begin{lemma}
   \label{lem:22}
   Let $\Phi=(E,\{f_i\}_{i=1}^m)\in\mathcal{E}$. Then for any $k\in\N$ we have
   \[\dim_H\mathcal{U}_k(\Phi)\le \dim_H\mathcal{U}_1(\Phi).\]
 \end{lemma}
\begin{proof}
  Take  $x\in\mathcal{U}_k(\Phi)$. Then all  codings of $x$ will eventually end in
  $$\pi^{-1}(\mathcal{U}_1(\Phi))=\{(c_i)\in\{1,\cdots,m\}^{\infty}: \pi((c_i))\in\mathcal{U}_1(\Phi)\}.$$
  In other words,
    \[
  \mathcal{U}_k(\Phi)\subseteq \bigcup_{d_1\cdots d_n\in\{1,2,\cdots,m\}^*}f_{d_1\cdots d_n}(\mathcal{U}_1(\Phi)).
  \]
  Therefore, the lemma follows by the countable stability of Hausdorff dimension (cf.~\cite{Falconer_1990}).
\end{proof}

For the lower bound of $\dim_H\mathcal{U}_k(E)$ we  split the proof into the following two  subsections.
\[(i)~  f_1(I)\cap f_2(I)\ne \emptyset\textrm{ or }f_{m-1}(I)\cap f_m(I)\ne\emptyset;\quad (ii)  ~
f_1(I)\cap f_2(I)=f_{m-1}(I)\cap f_m(I)=\emptyset.\]

\bigskip

\subsection{ $f_1(I)\cap f_2(I)\ne \emptyset$ or $f_{m-1}(I)\cap f_m(I)\ne\emptyset$}
Let $\Phi=(E, \set{f_i}_{i=1}^m)\in\ee$ with the convex hull $I=conv(E)$. By Condition (C) there exists $i_0\in\set{1,\ldots,m-1}$ such  that
 $f_{i_0}(I)\cap f_{i_0+1}(I)=\emptyset$.
In the following lemma we  show that the Hausdorff dimension of $\mathcal{U}_1(\Phi)$ is dominated by the subset consisting of all $x\in\mathcal{U}_1(E)$ with its unique coding starting at digit $i_0$ or $i_0+1$.
\begin{lemma}
  \label{lem:27}
  Let $\Phi=(E,\{f_i\}_{i=1}^m)\in\mathcal{E}$ with $I=conv(E)$. If $f_{i_0}(I)\cap f_{i_0+1}(I)=\emptyset$ for some $i_0\in\{1,\cdots,m-1\}$, then
  { \[
  \dim_H\mathcal{U}_1(\Phi)= \dim_H (f_{i_0}(E)\cap \mathcal{U}_1(\Phi))=\dim_H (f_{i_0+1}(E)\cap \mathcal{U}_1(\Phi)).
  \]}
\end{lemma}
\begin{proof}
  Note that $\mathcal{U}_1(\Phi)=\bigcup_{j=1}^m (f_j(E)\cap \mathcal{U}_1(\Phi))$. It suffices to prove
  \begin{equation}\label{eq:d1}
  \dim_H (f_{i_0}(E)\cap \mathcal{U}_1(\Phi))\ge\dim_H\left(\bigcup_{j=i_0+1}^m f_j(E)\cap \mathcal{U}_1(\Phi)\right)
  \end{equation}
  and
  \[
   \dim_H (f_{i_0+1}(E)\cap \mathcal{U}_1(\Phi))\ge\dim_H\left(\bigcup_{j=1}^{i_0} f_j(E)\cap \mathcal{U}_1(\Phi)\right).
  \]
 Without loss of generality we only prove (\ref{eq:d1}). Let
  \begin{align*}
    \varphi:\quad \bigcup_{j=i_0+1}^m f_j(E)\cap \mathcal{U}_1(\Phi)&\longrightarrow f_{i_0}(E)\cap \mathcal{U}_1(\Phi)\\
    x\quad\qquad& \mapsto\quad\quad f_{i_0}(x).
  \end{align*}
 First we prove that $\varphi$ is well-defined. Take $x\in\bigcup_{j=i_0+1}^m f_j(E)\cap \mathcal{U}_1(\Phi)$. It suffices to prove that $f_{i_0}(x)\in\mathcal{U}_1(\Phi)$. Suppose
 on the contrary that $f_{i_0}(x)\notin\mathcal{U}_1(\Phi)$. Note that $f_{i_0}(I)\cap f_{i_0+1}(I)=\emptyset$. Then by the locations of the fundamental intervals  it follows that
 \[
 f_{i_0}(x)\in f_{i_0-1}(I)\cap f_{i_0}(I) \subseteq f_{i_0 1}(I).
 \]
  This implies that $x\in f_1(I)$, leading to a contradiction with $f_1(I)\cap \bigcup_{j=i_0+1}^m f_j(I)=\emptyset$.

  Therefore, $\varphi$ is well-defined. Note that $\varphi$ is a similitude.   Hence, by \cite[Proposition 3.3]{Falconer_1990} it follows that
  \begin{align*}
    \dim_H \big(f_{i_0}(E)\cap \mathcal{U}_1(\Phi)\big)&\ge\dim_H\varphi\left(\bigcup_{j=i_0+1}^m f_j(E)\cap \mathcal{U}_1(\Phi)\right)\\
    & =\dim_H\left(\bigcup_{j=i_0+1}^m f_j(E)\cap \mathcal{U}_1(\Phi)\right).
  \end{align*}
  This establishes (\ref{eq:d1}).
\end{proof}

\begin{lemma}\label{lem:finite-k}
Let $\Phi=(E,\set{f_i}_{i=1}^m)\in\ee$ with $I=conv(E)$, and let $k\in\N$.
\begin{enumerate}[{\rm(i)}]
\item If $f_1(I)\cap f_2(I)=f_{1m^u}(I)=f_{21^v}(I)$ and $f_{i_0}(I)\cap f_{i_0+1}(I)=\emptyset$ for some $i_0\in\set{2,\ldots, m-1}$, then
\[
\pi(1m^{u(k-1)}\c)\in\uu_k(\Phi)
\]
for any $\pi(\c)\in f_{i_0}(E)\cap \uu_1(\Phi)$.

\item If $f_{m-1}(I)\cap f_m(I)=f_{(m-1)m^u}(I)=f_{m1^v}(I)$ and $f_{i_0}(I)\cap f_{i_0+1}(I)=\emptyset$ for some $i_0\in\set{1,\ldots, m-2}$, then
\[
\pi(m1^{v(k-1)}\c)\in\uu_k(\Phi)
\]
for any $\pi(\c)\in f_{i_0+1}(E)\cap \uu_1(\Phi)$.
\end{enumerate}
\end{lemma}
\begin{proof}
Since the proof of (ii) is similar, we only prove (i).

Take $\pi(\bm{c}) =\pi((c_i))\in f_{i_0}(E)\cap\mathcal{U}_1(\Phi)$. Then $c_1=i_0$. It suffices to prove  that for any integer $k\ge 0$,
\[
z_k:=\pi(1m^{uk}\bm{c})
\]
has exactly $k+1$ different codings. We will prove this by induction on $k$.
Suppose $k=0$. Then $z_0=\pi(1\bm{c})$. Note that $f_{i_0}(I)\cap f_{i_0+1}(I)=\emptyset$ for some $i_0\in\{2,\cdots,m-1\}$. Denote by $I=[a, b]$. Then by   Condition (A) it follows that
\begin{align*}
z_0&= \pi(1 i_0 c_2c_3\cdots)  \le f_{1i_0}(b)<f_{1m^u}(a)=f_{21^v}(a)=f_2(a),
 \end{align*}
 where the equality $f_{1m^u}(a)=f_{21^v}(a)$ follows by Lemma \ref{lem:21}.
 This together with $\pi(\bm{c})=\pi(i_0 c_2c_3\cdots)\in\mathcal{U}_1(\Phi)$ implies that $z_0\in\uu_1(\Phi)$.

Now suppose that $z_k$ has $k+1$ different codings for some $k\ge 0$. We will prove that $z_{k+1}$ has exactly $k+2$ different codings. Note   that
\begin{equation}\label{eq:d3}
\begin{split}
  z_{k+1}&=f_{1m^u}(\pi(m^{uk}\bm{c}))
  =f_{21^v}(\pi(m^{uk}\bm{c}))=f_{21^{v-1}}(z_k).
\end{split}
\end{equation}
By the induction hypothesis this implies that $z_{k+1}$ has at least $k+2$ different codings: one is $1m^{u(k+1)}\bm{c}$, and the others start at $21^{v-1}$.
In the following
we show   that $z_{k+1}$ has precisely   $k+2$ different codings. Suppose  $(d_i)$ is a coding of $z_{k+1}$. Then by (\ref{eq:d3}) and Lemma \ref{lem:key lem} it follows that
\[
d_1\cdots d_u=1m^{u-1}\quad\textrm{or}\quad d_1\cdots d_v=21^{v-1}.
\]
By the induction hypothesis  it suffices to prove that $d_1\cdots d_u=1m^{u-1}$ implies $(d_i)=1m^{u(k+1)}\bm{c}$.

 {Suppose $d_1\cdots d_u=1m^{u-1}$. We claim that $d_{u+1}\cdots d_{u(k+1)+1}=m^{uk+1}$. Let $1\le j\le uk+1$ be the smallest integer such that
  $d_{u+j}\ne m$. Then by
  (\ref{eq:d3}) we have
 \begin{equation}\label{eq:k6}
\pi( d_{u+j}d_{u+j+1}\cdots ) =f_{ m^{uk+2-j}}(\pi(\bm{c})).
\end{equation}
In view of the locations of the basic intervals we have $d_{u+j}=m-1$. Therefore, by (\ref{eq:k6}) and Condition (D) it follows  that
\[
f_{m^{uk+2-j}}(\pi(\bm{c}))\in f_{m-1}(I)\cap f_m(I) \subseteq f_{m1}(I).
\]
This implies that $f_{m^{uk+1-j}}(\pi(\bm{c}))\in f_1(I)$. If $j<uk+1$, then we obtain $f_m(I)\cap f_1(I)\ne\emptyset$, leading to a contradiction with  $f_{i_0}(I)\cap f_{i_0+1}(I)=\emptyset$. If $j=uk+1$, then  we get $\pi(\bm{c})\in f_1(I)$, leading to a contradiction with
  $\pi(\bm{c})=\pi(i_0 c_2c_3\cdots)\in\mathcal{U}_1(\Phi)$ and $i_0\ge 2$.  Thus,    $d_{u+1}\ldots d_{u(k+1)+1}=m^{uk+1}$. Since $\pi(\mathbf c)\in\uu_1(\Phi)$, we conclude that  $(d_i)=1m^{u(k+1)}\bm{c}$.}

Hence,   by induction it follows that $z_k$ has exactly $k+1$ different codings for any integer $k\ge 0$. This establishes (i).
\end{proof}

\begin{lemma}
  \label{lem:28}
  Let $\Phi=(E,\{f_i\}_{i=1}^m)\in\mathcal{E}$ with $I=conv(E)$. If $f_1(I)\cap f_2(I)\ne\emptyset$ or $f_{m-1}(I)\cap f_{m}(I)\ne\emptyset$, then
 \[
  \dim_H\mathcal{U}_k(E)\ge \dim_H\mathcal{U}_1(E)\quad\forall ~k\in\N.
  \]
\end{lemma}
\begin{proof}
Without loss of generality we assume $f_1(I)\cap f_2(I)\ne \emptyset$. By Condition (C) there exists $i_0\in\set{2,\ldots, m-1}$ such that
  $f_{i_0}(I)\cap f_{i_0+1}(I)=\emptyset$. Then by Lemma \ref{lem:27} it follows that
    \begin{equation}
    \label{eq:d2}
    \dim_H\mathcal{U}_1(\Phi)=\dim_H (f_{i_0}(E)\cap \mathcal{U}_1(\Phi)).
  \end{equation}
  Note that $f_1(I)\cap f_2(I)\ne\emptyset$. Then by Condition (D) there exist $u=u_1, v=v_1\in\N$ such that
$
f_1(I)\cap f_2(I)=f_{1m^u}(I)=f_{21^v}(I).
$
 So by Lemma \ref{lem:finite-k} (i) it follows that
\[
\{\pi(1m^{u(k-1)}\bm{c}): \pi(\bm{c})\in f_{i_0}(E)\cap\mathcal{U}_1(\Phi)\}\subseteq\mathcal{U}_{k}(\Phi).
\]
Hence, by (\ref{eq:d2}) we conclude that
\[\dim_H\mathcal{U}_{k}(\Phi)\ge \dim_H (f_{i_0}(E)\cap \mathcal{U}_1(\Phi))=\dim_H\mathcal{U}_1(\Phi). \]
\end{proof}
\subsection{$f_1(I)\cap f_2(I)=f_{m-1}(I)\cap f_m(I)=\emptyset$.} Let $\Phi=(E,\{f_i\}_{i=1}^m)\in\mathcal{E}$. Different from the previous subsection it happens in this case that $\uu_k(\Phi)=\emptyset$ for some $k\in\N$.
\begin{lemma}
  \label{lem:29}
  Let $\Phi=(E,\{f_i\}_{i=1}^m)\in\mathcal{E}$ with $I=conv(E)$. If
  $f_1(I)\cap f_2(I)=f_{m-1}(I)\cap f_m(I)=\emptyset,$
  then
  \[\mathcal{U}_k(\Phi)=\emptyset\quad \textrm{for any } k\ne 2^\ell\textrm{ with  }\ell\in \mathbb{N}\cup\set{0}.
  \]
\end{lemma}
\begin{proof}
For $x\in E$ we denote by $N(x)$ the number of different codings of $x$ with respect to $\{f_i\}_{i=1}^m$.
  Let $k\ge 2$ and take $x\in\mathcal{U}_k(\Phi)$.  Then $N(x)=k$.
  So, there exists
  \begin{equation}
\label{eq:r1}
x_1\in\mathcal{U}_k(\Phi)\cap f_{i}(I)\cap f_{i+1}(I)
\end{equation}
for some $i\in\set{2,\ldots, m-2}$
   such that  $N(x)=N(x_1)$.
   Note that $f_i({  I})\cap f_{i+1}({  I})\ne\emptyset$. Then by Condition (D) there exist $u=u_i, v=v_i\in\N$ such that
\begin{equation}\label{eq:r2}
f_{i}(I)\cap f_{i+1}(I)=f_{im^u}(I)=f_{(i+1)1^v}(I).
\end{equation}
Observe that $f_1(I)\cap f_2(I)= f_{m-1}(I)\cap f_m(I)=\emptyset$. Then in view of the locations of these basic intervals we obtain
\begin{equation}\label{eq:r7}
f_{1}(I)\cap f_i(I)=\emptyset~~ \forall i\ne 1\quad\textrm{and}\quad f_j(I)\cap f_m(I)=\emptyset~~\forall j\ne m.
\end{equation}

Therefore, by (\ref{eq:r1})--(\ref{eq:r7}) and a similar argument as in the proof of Lemma \ref{lem:key lem}  it follows that all codings of $x_1$ either begin with $im^{u}$ or begin with $(i+1)1^{v}$. By (\ref{eq:r2}) and Lemma \ref{lem:21} it gives that $f_{i m^u}(\cdot)=f_{(i+1)1^v}(\cdot)$. Then
there exists a unique
$y\in E$ such that
$
 x_1 =  f_{i  m^u}(y)=  f_{(i +1)1^v}(y).
$
Furthermore,
\[
N(x_1)=N(f_{m^u}(y))+N(f_{1^v}(y))=2N(y),
\]
where the last equality holds  by (\ref{eq:r7}) that
 \[N(f_{m^u}(y))=N(y)=N(f_{1^v}(y)).\]
Hence, we conclude that $N(x)=N(x_1)=2N(y)$.

  By iteration of the above arguments it follows that $N(x)$ must be of the form $2^\ell$ for some $\ell\in\N\cup\set{0}$. This completes the proof.
\end{proof}
In fact the condition $k\ne 2^{\ell}$ in the above lemma is also a necessary   condition for $\uu_k(\Phi)=\emptyset$.
\begin{lemma}
  \label{lem:210}
  Let $\Phi=(E,\{f_i\}_{i=1}^m)\in\mathcal{E}$ with $I=conv(E)$. If $f_1(I)\cap f_2(I)=f_{m-1}(I)\cap f_m(I)=\emptyset,$ then
  \[\dim_H\mathcal{U}_{2^\ell}(\Phi)\ge\dim_H\mathcal{U}_1(\Phi)\quad\forall \ell\in\N\cup\set{0}.\]
\end{lemma}
\begin{proof}
  We will prove the inequality $\dim_H\mathcal{U}_{2^\ell}(\Phi)\ge \dim_H\mathcal{U}_1(\Phi)$   by induction on $\ell$. Trivially this inequality holds  if $\ell=0$.
 Now we assume this inequality holds for some $\ell\ge 0$. It suffices to prove
  \[
  \dim_H\mathcal{U}_{2^{\ell+1}}(\Phi)\ge\dim_H\mathcal{U}_{2^\ell}(\Phi).
  \]
 By Condition (C) there exists $i\in\{2,\cdots,m-2\}$ such that  $f_{i}(I)\cap f_{i+1}(I)\ne\emptyset$. Then  Condition (D) gives two indexes  $u=u_i, v=v_i\in\N$ such that
  \[
  f_i(I)\cap f_{i+1}(I)=f_{im^u}(I)=f_{(i+1)1^v}(I).
  \]
  By Lemma \ref{lem:21} this yields that  $f_{im^u}(\cdot)=f_{(i+1)1^v}(\cdot)$. Note that any three basic intervals have an empty intersection. So, by (\ref{eq:r7}) it follows that
  \[
 \{ f_{im^u}(x)=f_{(i+1)1^v}(x): x\in\mathcal{U}_{2^\ell}(\Phi)\}\subseteq\mathcal{U}_{2^{\ell+1}}(\Phi),
  \]
which implies $\dim_H\mathcal{U}_{2^{\ell+1}}(\Phi)\ge \dim_H\mathcal{U}_{2^\ell}(\Phi)$.

  By induction we conclude that
$
  \dim_H\mathcal{U}_{2^\ell}(\Phi) \ge \dim_H\mathcal{U}_1(\Phi)
$ for any $\ell\in\N$.
\end{proof}

\begin{proof}[{Proof of  $(i) \Leftrightarrow (ii)$ in Theorems \ref{th:1}
(A) and (B)}] Let $\Phi=(E, \set{f_i}_{i=1}^m)\in\ee$ with $I=conv(E)$.
 By Lemmas \ref{lem:22} and \ref{lem:28} it follows that if $f_1(I)\cap f_2(I)\ne\emptyset$ or $f_{m-1}(I)\cap f_m(I)\ne\emptyset$, then
 \[
 \dim_H\mathcal{U}_k(\Phi)=\dim_H\mathcal{U}_1(\Phi)\quad\textrm{for all}\quad k\in\N.\]
 On the other hand, if $f_1(I)\cap f_2(I)=f_{m-1}(I)\cap f_m(I)=\emptyset$, then by Lemmas \ref{lem:22}, \ref{lem:29} and \ref{lem:210} it follows that
\[
\left\{\begin{array}{ll}
 \dim_H\mathcal{U}_k(\Phi)=\dim_H\mathcal{U}_1(\Phi)&\textrm{if}\quad k=2^\ell\textrm{ with }\ell\in\N\cup\set{0},\\
 \mathcal{U}_k(\Phi)=\emptyset&\textrm{otherwise}.
 \end{array}
 \right.
 \]
 Therefore, the equivalence (i) $\Leftrightarrow$ (ii) holds ture.
\end{proof}

\section{Countably many codings}\label{sec:countable codings}
Given $\Phi=(E, \set{f_i}_{i=1}^m)\in\ee$, we will consider in this section the set $\mathcal{U}_{\aleph_0}(\Phi)$ of points having countably many   codings, and prove the equivalences $(i)\Leftrightarrow(iii)\Leftrightarrow(iv)$ in Theorem \ref{th:1} (A) and (B), respectively. The equivalence $(i)\Leftrightarrow(iii)$ follows from the following result.
\begin{proposition}
 \label{lem:31}
 Let $\Phi=(E,\{f_i\}_{i=1}^m)\in\mathcal{E}$ with $I=conv(E)=[a, b]$.
 \begin{itemize}
 \item [{\rm(i)}]  $f_1(I)\cap f_2(I)\ne\emptyset$ if and only if $f_1(b)\in\mathcal{U}_{\aleph_0}(\Phi)$.
 \item [{\rm(ii)}] $f_{m-1}(I)\cap f_m(I)\ne\emptyset$ if and only if $f_m(a)\in\mathcal{U}_{\aleph_0}(\Phi)$.
 \end{itemize}
\end{proposition}
\begin{proof}
Since the proofs of (i) and (ii) are similar, we only prove (i).

Suppose   $f_1(I)\cap f_2(I)=\emptyset$. Then in view of  the locations of these basic intervals we have
  \begin{equation}\label{eq:sep-countable-1}
  f_1(I)\cap f_i(I)=\emptyset \quad\textrm{for any}\quad i\in\set{2,\ldots, m}.
  \end{equation}
 Note that $b=\pi(m^{\infty})\in\mathcal{U}_1(\Phi)$. Then by (\ref{eq:sep-countable-1}) it follows that  $f_1(b)\in\mathcal{U}_1(\Phi)$. This proves the sufficiency in (i).

For  the necessity in (i) we assume   $f_1(I)\cap f_2(I)\ne\emptyset$. Then by Condition (D) there exist $u=u_1, v=v_1\in\N$ such that
 $
 f_1(I)\cap f_2(I)=f_{1m^u}(I)=f_{21^v}(I),
 $
which  gives $f_{1m^u}(\cdot)=f_{21^v}(\cdot)$ by Lemma \ref{lem:21}.
Then
 \begin{equation}\label{eq:r3}
 \begin{split}
 f_1(b)&=\pi(1m^{\infty})=\pi(21^{v-1}1m^{\infty}) =\cdots=\pi((21^{v-1})^k1 m^{\infty})=\cdots
 \end{split}
 \end{equation}
for $k=1,2,\ldots$.
 This implies that $f_1(b)$ has at least countably  many codings.

In the following we show that $f_1(b)$ indeed has countably  many codings.
Suppose that $(d_i)$ is a coding of $f_1(b)$. Note that $f_1(b)\in f_1(I)\cap f_2(I)=f_{1m^u}(I)=f_{21^v}(I)$. Then by  Lemma \ref{lem:key lem} it follows  that $d_1\cdots d_u=1m^{u-1}$ or $d_1\cdots d_v=21^{v-1}$.

\begin{itemize}
  \item If $d_1\cdots d_u=1m^{u-1}$, then by (\ref{eq:r3}) we have
  \[
  \pi(d_{u+1}d_{u+2}\cdots)=\pi(m^{\infty})\in\mathcal{U}_1(\Phi).
  \]
  This implies that $(d_i)=1m^{\infty}$.

  \item If $d_1\cdots d_v=21^{v-1}$, then by (\ref{eq:r3}) it yields that
  \begin{equation*}
  \pi(d_{v+1}d_{v+2}\cdots)=\pi(1m^{\infty})=f_1(b).
  \end{equation*}
\end{itemize}

By iteration of the above arguments it follows that all  codings of $f_1(b)$ are of the form
 \[
 (21^{v-1})^k 1m^{\infty},\quad k\ge 0.
 \]
 Hence, $f_1(b)\in\mathcal{U}_{\aleph_0}(\Phi)$. This proves the necessity in (i).
\end{proof}
\begin{proof}
  [{Proof of (i) $\Leftrightarrow$ (iii) in Theorem \ref{th:1} (A) and (B)}] The equivalence  (i) $\Leftrightarrow$ (iii) in Theorem \ref{th:1} (A) and (B) follows directly from Proposition \ref{lem:31}.
\end{proof}
In the following it remains to prove the equivalence (i) $\Leftrightarrow$ (iv) in Theorem \ref{th:1} (A) and (B).
\begin{lemma}
  \label{lem:32}
  Let $\Phi=(E,\{f_i\}_{i=1}^m)\in\mathcal{E}$ with $I=conv(E)$. If
  $f_1(I)\cap f_2(I)=f_{m-1}(I)\cap f_m(I)=\emptyset,$
  then $\mathcal{U}_{\aleph_0}(\Phi)=\emptyset$.
\end{lemma}
\begin{proof}
  Suppose on the contrary that $\mathcal{U}_{\aleph_0}(\Phi)\ne\emptyset$, and take $x\in\mathcal{U}_{\aleph_0}(\Phi)$. Then $x$ must have a coding $(d_i)$ satisfying
  \begin{equation}
    \label{eq:c1}
    x_n:=\pi(d_{n+1}d_{n+2}\cdots)\in E\cap\bigcup_{i=2}^{m-2} \big(f_i(I)\cap f_{i+1}(I)\big)
  \end{equation}
  for infinitely many $n\ge 1$.

  Take $n$ satisfying (\ref{eq:c1}) and assume that
  $
 x_n\in E\cap f_{i}(I)\cap f_{i+1}(I)
  $
for some $2\le i\le m-2$. By Condition (D) there exist $u=u_i, v=v_i\in\N$ such that
\begin{equation}\label{eq:c2}
\begin{split}
x_n=\pi(d_{n+1}d_{n+2}\ldots)&\in f_i(I)\cap f_{i+1}(I)=f_{im^u}(I)=f_{(i+1)1^v}(I).
\end{split}
\end{equation}
Note that $f_1(I)\cap f_2(I)=f_{m-1}(I)\cap f_m(I)=\emptyset$. Then in view of    the locations of the basic intervals it follows that
\begin{equation}
  \label{eq:c3}
  f_1(I)\cap f_j(I)=f_\ell(I)\cap f_m(I)=\emptyset
\end{equation}
for any $j\ne 1$ and any $\ell\ne m$. Therefore, by (\ref{eq:c2}) and (\ref{eq:c3}) it follows that
\[
d_{n+1}\cdots d_{n+u+1}=im^u \quad\textrm{or}\quad d_{n+1}\cdots d_{n+v+1}=(i+1)1^v.
\]
Note by   (\ref{eq:c2}) and Lemma \ref{lem:21} that $f_{im^u}(\cdot)=f_{(i+1)1^v}(\cdot)$. Therefore, we have a substitution $im^u\sim (i+1)1^v$ in $d_{n+1}d_{n+2}\cdots$. In other words, there exists a unique $y\in E$ such that $x_n=f_{im^u}(y)=f_{(i+1)1^v}(y)$.

Note that  (\ref{eq:c1})  holds for infinitely many $n\in\N$. Then by a similar argument as above    it follows that there exist infinitely many independent substitutions  in $(d_i)$.  This implies that $x$ has a continuum of codings, leading to a contradiction with $x\in\mathcal{U}_{\aleph_0}(\Phi)$.
\end{proof}

\begin{lemma}
  \label{lem:33}
  Let $\Phi=(E,\{f_i\}_{i=1}^m)\in\mathcal{E}$ with $I=conv(E)$.
  \begin{enumerate}[{\rm(i)}]
  \item If  $f_1(I)\cap f_2(I)=f_{1m^u}(I)=f_{21^v}(I)$ and $f_{m-1}(I)\cap f_m(I)=\emptyset$, then $|\mathcal{U}_{\aleph_0}(\Phi)|=\aleph_0$. Furthermore, any  point in $\uu_{\aleph_0}(\Phi)$   must have a coding ending with $1m^\f\sim (21^{v-1})^\f$.

  \item If  $f_1(I)\cap f_2(I)=\emptyset$ and  $f_{m-1}(I)\cap f_m(I)=f_{(m-1)m^p}(I)=f_{m1^q}(I)$, then $|\mathcal{U}_{\aleph_0}(\Phi)|=\aleph_0$. Furthermore, any  point in $\uu_{\aleph_0}(\Phi)$   must have a coding ending with $m1^\f\sim ((m-1)m^{p-1})^\f$.

  \item If $f_1(I)\cap f_2(I)=f_{1m^{u}}(I)=f_{21^{v}}(I)$ and $f_{m-1}(I)\cap f_m(I)=f_{(m-1)m^{p}}(I)=f_{m1^{q}}(I)$, then $|\mathcal{U}_{\aleph_0}(\Phi)|=\aleph_0$. Furthermore, any  point in $\uu_{\aleph_0}(\Phi)$   must have a coding ending with $1m^\f\sim (21^{v-1})^\f$ or $m1^\f\sim ((m-1)m^{p-1})^\f$.
  \end{enumerate}
\end{lemma}
\begin{proof}
Since the proofs of (ii) and (iii) are similar, we only prove (i).  Suppose  $f_1(I)\cap f_2(I)\ne\emptyset$ but $f_{m-1}(I)\cap f_m(I)=\emptyset$. By Condition (D) there exist $u=u_1, v=v_1\in\N$ such that
  \begin{equation}
    \label{eq:b1}
    f_1(I)\cap f_2(I)=f_{1m^{u}}(I)=f_{21^{v}}(I),
  \end{equation}
which yields   $f_{1m^u}(\cdot)=f_{21^v}(\cdot)$  by Lemma \ref{lem:21}.

Denote by $I=[a, b]$ the convex hull of $E$. First we claim  $f_{1^n}(b)\in\mathcal{U}_{\aleph_0}(\Phi)$ for any $n\ge 1$. We  will prove this by induction on $n$. Clearly, for $n=1$ we have by Lemma \ref{lem:31} that $f_1(b)\in\mathcal{U}_{\aleph_0}(\Phi)$.
Now suppose   $f_{1^n}(b)\in\mathcal{U}_{\aleph_0}(\Phi)$ for some $n\ge 1$, and we consider $f_{1^{n+1}}(b)$.    If $f_{1^{n+1}}(b)\in f_2(I)$,  then by Condition (D) it follows that
\[
f_{1^{n+1}}(b)\in f_1(I)\cap f_2(I)\subseteq f_{1m}(I),
\]
which implies $f_{1^n}({  b})\in f_m(I)$, leading to a contradiction with $f_1(I)\cap f_m(I)=\emptyset$. By the induction hypothesis this proves $f_{1^{n+1}}(b)\in\uu_{\aleph_0}(\Phi)$. Hence,
by induction it follows that  $\{f_{1^n}(b): n\ge 1\}\subseteq\mathcal{U}_{\aleph_0}(\Phi)$.

In the following it suffices to prove that any  point in $\uu_{\aleph_0}(\Phi)$   must have a coding ending with $1m^{\infty}\sim (21^{v-1})^{\infty}$.
Take $x\in\mathcal{U}_{\aleph_0}(\Phi)$. Suppose on the contrary that   $x$ has a coding $(d_i)$  ending with  neither $1m^{\infty}$ nor $ (21^{v-1})^{\infty}$. Observe that  $(d_i)$ satisfies
  \begin{equation}\label{eq:b2}
 \pi(d_{n+1}d_{n+2}\cdots) \in E\cap \bigcup_{i=1}^{m-2}\big(f_i(I)\cap f_{i+1}(I)\big)
  \end{equation}
  for infinitely many $n\ge 1$.

Take $n$ satisfying (\ref{eq:b2}), and assume   $\pi(d_{n+1}d_{n+2}\cdots)\in  f_i(I)\cap f_{i+1}(I)$ for some $i\in\{1,\cdots,m-2\}$. Then by Condition (D) there exist $p=u_i, q=v_i\in\N$ such that
  \begin{equation}\label{eq:a3}
  \pi(d_{n+1}d_{n+2}\cdots)\in f_i(I)\cap f_{i+1}(I)=f_{im^p}(I)=f_{(i+1)1^q}(I).
  \end{equation}
  By Lemma  \ref{lem:21}  it follows that $f_{im^p}(\cdot)=f_{(i+1)1^q}(\cdot)$, and then by Lemma  \ref{lem:key lem} we have
   \[d_{n+1}\cdots d_{n+p}=im^{p-1}\quad\textrm{ or}\quad d_{n+1}\cdots d_{n+q}=(i+1)1^{q-1}.\]
 Now we split the proof into the following two cases.

  Case (I). $d_{n+1}\cdots d_{n+p}=im^{p-1}$. Then by (\ref{eq:a3}) and using $f_{m-1}(I)\cap f_m(I)=\emptyset$ it follows that $d_{n+p+1}=m$.
Therefore,  we have a substitution    $d_{n+1}\cdots d_{n+p+1}=im^{p}\sim (i+1)1^{q}$.

  Case (II). $d_{n+1}\cdots d_{n+q}=(i+1)1^{q-1}$. Then by (\ref{eq:a3}) it follows that
 $d_{n+q+1}=1$ or $2$.
 \begin{itemize}
 \item
  If $d_{n+q+1}=1$, then we have a substitution $d_{n+1}\cdots d_{n+q+1}=(i+1)1^q\sim im^{p}.$

  \item If $d_{n+q+1}=2$, then by (\ref{eq:b1}) and (\ref{eq:a3}) it yields that
  \[\pi(d_{n+q+1}d_{n+q+2}\cdots)\in f_1(I)\cap f_2(I)=f_{1m^u}(I)=f_{{  2}1^v}(I).\]
  So, by Lemma \ref{lem:key lem} it follows that
  $
  d_{n+q+1}\cdots d_{n+q+v}=21^{v-1}
  $ and $d_{n+q+v+1}=1$ or $2$.
   Note by the assumption  that $(d_i)$ does not end with $(21^{v-1})^{\infty}$. Then by iteration it follows that   there exists $N\ge n+q$ such that
\[
  d_{N+1}\cdots d_{N+v+1}=21^{v}.
\]
\end{itemize}
Again, we  have a substitution  $21^v\sim 1m^{u}$.

  By Cases (I)--(II) and (\ref{eq:b2}) it follows that there exist  infinitely many independent substitutions in $(d_i)$. This implies that $x$ has a continuum of codings, leading to a contradiction with $x\in\mathcal{U}_{\aleph_0}(\Phi)$.  We complete the proof of (i).
  \end{proof}

  By the proof of Lemma \ref{lem:33} it follows that if $f_1(I)\cap f_2(I)\ne\emptyset$, then   any $x\in\uu_{\aleph_0}(\Phi)$ must have a coding  ending  with $1m^\f$, which is a  coding of $f_1(b)$. Similarly, if $f_{m-1}(I)\cap f_m(I)\ne\emptyset$, then  any $x\in\uu_{\aleph_0}(\Phi)$ must have a  coding ending with $m 1^\f$, which is a coding of $f_m(a)$.
As a corollary of Lemma \ref{lem:33} we have a complete characterization of $\uu_{\aleph_0}(\Phi)$.
\begin{proposition}\label{cor:countable}
Let $\Phi=(E, \set{f_i}_{i=1}^m)\in\ee$ with  $I=conv(E)$. Then
\[
\uu_{\aleph_0}(\Phi)=\left\{\begin{array}{lll}
\bigcup_{\mathbf i\in\set{1,\ldots, m}^*}\set{f_{\mathbf i}(f_1(b))}&\textrm{if}& f_1(I)\cap f_2(I)\ne\emptyset,~ f_{m-1}(I)\cap f_m(I)=\emptyset;\\
\bigcup_{\mathbf i\in\set{1,\ldots, m}^*}\set{f_{\mathbf i}(f_m(a))}&\textrm{if}& f_1(I)\cap f_2(I)=\emptyset,~ f_{m-1}(I)\cap f_m(I)\ne\emptyset;\\
\bigcup_{\mathbf i\in\set{1,\ldots,m}^*}\set{f_{\mathbf i}(f_1(b)), f_{\mathbf i}(f_m(a))}&\textrm{if}& f_1(I)\cap f_2(I)\ne\emptyset,~ f_{m-1}(I)\cap f_m(I)\ne\emptyset.
\end{array}\right.
\]
\end{proposition}

\begin{proof}
  [{Proof of (i) $\Leftrightarrow$ (iv) in Theorem \ref{th:1} (A) and (B)}]
  The equivalence  (i) $\Leftrightarrow$ (iv) follows directly from Lemmas \ref{lem:32} and \ref{lem:33}.
  \end{proof}

\section{Topology of $\mathcal{U}_k(\Phi)$}\label{sec:closeness}
In this section we investigate the topology of   $\mathcal{U}_k(\Phi)$, and prove the equivalence (i) $\Leftrightarrow$ (v) in Theorem \ref{th:1} (A) and (B) respectively.  First  we prove the following useful lemma.

\begin{lemma}\label{lem:colse-1}
Let $\Phi=(E, \set{f_i}_{i=1}^m)\in\ee$ with $I=conv(E)$.
\begin{enumerate}[{\rm(i)}]
\item If $f_1(I)\cap f_2(I)=f_{1m^u}(I)=f_{21^v}(I)$ and $f_{i_0}(I)\cap f_{i_0+1}(I)=\emptyset$ for some $i_0\in\set{2,\ldots, m-1}$, then for any integer $\ell\ge 0$
\[
\pi(i_0m(21^{v-1})^\ell\mathbf c)\in\uu_1(\Phi),
\]
where  $\pi(\mathbf c)\in f_{i_0+1}(E)\cap \uu_1(\Phi)$.

\item  If $f_{m-1}(I)\cap f_m(I)=f_{(m-1)m^u}(I)=f_{m1^v}(I)$ and $f_{i_0}(I)\cap f_{i_0+1}(I)=\emptyset$ for some $i_0\in\set{1,\ldots, m-2}$, then for any integer $\ell\ge 0$
\[
\pi((i_0+1) 1((m-1)m^{u-1})^\ell\mathbf c)\in\uu_1(\Phi),
\]
where  $\pi(\mathbf c)\in f_{i_0}(E)\cap \uu_1(\Phi)$.
\end{enumerate}
\end{lemma}
\begin{proof}
Since the proof of (ii) is similar, we only prove (i).

Let $x_\ell:=\pi(i_0m(21^{v-1})^\ell \c)$. We will prove $x_\ell\in\uu_1(\Phi)$ by induction on $\ell$. First we prove $x_0=\pi(i_0m\c)\in\uu_1(\Phi)$. Note that $f_{i_0}(I)\cap f_{i_0+1}(I)=\emptyset$. Then in view of the location of the basic intervals it follows that
\begin{equation}\label{eq:location-1}
f_{i}(I)\cap f_{j}(I)=\emptyset\quad\forall~ i\le i_0\textrm{ and } j\ge i_0+1.
\end{equation}
Let $(d_i)$ be a coding of $x_0$.  We will prove $(d_i)=i_0m\c$, where $\c\in\pi^{-1}(f_{i_0+1}(E)\cap\uu_1(\Phi))$ has a prefix $i_0+1$. Since $\pi(\c)\in\uu_1(\Phi)$, it suffices to prove $d_1=i_0$ and $d_2=m$.

If $d_1\ne i_0$, then by (\ref{eq:location-1}) we must have $d_1=i_0-1$. By Condition (D) this implies
\[
\pi(i_0m\c)\in f_{i_0-1}(I)\cap f_{i_0}(I)\subset f_{i_01}(I),
\]
which gives $\pi(m\c)\in f_1(I)\cap f_m(I)$, leading to a contradiction with (\ref{eq:location-1}).
To determine $d_2$ suppose on the contrary  $d_2\ne m$. Then $d_2$ must be $m-1$. By Condition (D) it follows that
\[
\pi(m\c)\in f_{m-1}(I)\cap f_m(I)\subset f_{m 1}(I).
\]
This together with $\pi(\c)\in f_{i_0+1}(E)$ yields $\pi(\c)\in f_{1}(I)\cap f_{i_0+1}(I)$, again leading to a contradiction with (\ref{eq:location-1}). Therefore, $(d_i)=i_0m\c$ is the unique coding of $x_0$, and then $x_0\in\uu_1(\Phi)$.

We proceed with the induction hypothesis that $x_\ell=\pi(i_0 m(21^{v-1})^\ell\mathbf c)\in\uu_1(\Phi)$ for some $\ell\ge 0$, and we consider $x_{\ell+1}$. By the induction hypothesis it follows that
\begin{equation}\label{eq:close-2}
\pi((21^{v-1})^\ell \c)\in\uu_1(\Phi).
\end{equation}
Let $(d_i)$ be a coding of $x_{\ell+1}$. By (\ref{eq:close-2}) it suffices to prove
$d_1\ldots d_{v+2}=i_0m 21^{v-1}.$
 By the same argument as in the proof of $x_0\in\uu_1(\Phi)$ it follows that $d_1=i_0$. To determine the second digit $d_2$ we assume on the contrary that $d_2\ne m$. Then $d_2=m-1$, which implies
$
\pi(m(21^{v-1})^{\ell+1}\c)\in f_{m-1}(I)\cap f_m(I)\subset f_{m1}(I).
$
So
\begin{equation}\label{eq:may-29}
\pi((21^{v-1})^{\ell+1}\c)\in f_1(I)\cap f_2(I)=f_{21^v}(I).
\end{equation}
If $\ell\ge 1$, then (\ref{eq:may-29}) implies
\[
\pi((21^{v-1})^\ell\c)\in f_1(I)\cap f_2(I),
\]
leading to a contradiction with (\ref{eq:close-2}). If $\ell=0$, then (\ref{eq:may-29}) implies $\pi(\mathbf c)\in f_1(I)\cap f_{i_0+1}(I)$, leading to a contradiction with (\ref{eq:location-1}). Therefore, $d_2=m$.

To prove $d_3=2$ we consider the following two cases.
\begin{itemize}
\item $d_3=1$. Then $\pi((21^{v-1})^{\ell+1}\c)\in f_1(I)\cap f_2(I)=f_{21^v}(I)$. By the above arguments this will lead to a contradiction.

\item $d_3=3$. Then $\pi((21^{v-1})^{\ell+1}\c)\in f_2(I)\cap f_3(I)\subset f_{2m}(I)$, which gives $\pi(1^{v-1}(21^{v-1})^\ell\c)\in f_m(I)$. Again this leads to a contradiction with (\ref{eq:location-1}).
\end{itemize}
Therefore, $d_3=2$. If $v=1$, then by the induction hypothesis we are done.

If $v>1$, then we claim $d_4\ldots d_{v+2}=1^{v-1}$.
Suppose $d_4\ne 1$. Then $\pi(1^{v-1}(21^{v-1})^\ell\c)\in f_1(I)\cap f_2(I)\subset f_{1m}(I)$, which leads to a contradiction with (\ref{eq:location-1}). By iteration of this arguments we can deduce that $d_4=\cdots= d_{v+2}=1$.

Hence, $d_1\ldots d_{v+2}=i_0m21^{v-1}$. By the induction hypothesis we conclude that $(d_i)=im(21^{v-1})^{\ell+1}\c$, and thus $x_{\ell+1}\in\uu_1(\Phi)$,  completing  the proof.
\end{proof}

\begin{proposition}\label{prop:notclose}
Let $\Phi=(E,\{f_i\}_{i=1}^m)\in\mathcal{E}$ with $I=conv(E)$.
If $f_1(I)\cap f_2(I)\neq \emptyset$ or $f_{m-1}(I)\cap f_m(I)\ne \emptyset$,
then $\uu_k(\Phi)$ is not closed for any $k\in\N$.
\end{proposition}
\begin{proof}
Take $k\in\N$. We will prove that $\uu_k(\Phi)$ is not closed by considering  the following two cases.

(I). $f_1(I)\cap f_2(I)\ne \emptyset$. By Condition (D) there exist $u=u_1, v=v_1\in\N$ such that
$f_1(I)\cap f_2(I)=f_{1m^u}(I)=f_{21^v}(I)$. Furthermore, by Condition (C) there exists $i_0\in\set{2,\ldots, m-1}$ such that $f_{i_0}(I)\cap f_{i_0+1}(I)=\emptyset$. We claim that for any $\ell\in\N$
\[
x_{k,\ell}:=\pi(1m^{u(k-1)}\;i_0m(21^{v-1})^\ell\c)\in\uu_k(\Phi),
\]
where $\pi(\c)\in f_{i_0+1}(E)\cap \uu_1(\Phi)$. By Lemma \ref{lem:colse-1} (i) it follows that $\pi(i_0m(21^{v-1})^\ell\c)\in f_{i_0}(E)\cap \uu_1(\Phi)$. Then the claim follows by Lemma \ref{lem:finite-k} (i).

Observe that $x_{k,\ell}\in\uu_k(\Phi)$ converges to
\[
x_k=\lim_{\ell\to\f}z_{k,\ell}=\pi(1m^{u(k-1)}i_0m(21^{v-1})^\f),
\]
which has a coding ending with $(21^{v-1})^\f$. So by Proposition \ref{cor:countable} it follows that $x_k\in\uu_{\aleph_0}(\Phi)$. This implies that $\uu_k(\Phi)$ is not closed.

(II). $f_{m-1}(I)\cap f_m(I)\ne\emptyset$. Then there exist $u=u_{m-1}, v=v_{m-1}\in\N$ such that $f_{m-1}(I)\cap f_m(I)=f_{(m-1)m^u}(I)=f_{m1^v}(I)$. Furthermore, by Condition (C) there exists $i_0\in\set{1,\ldots, m-2}$ such that $f_{i_0}(I)\cap f_{i_0+1}(I)=\emptyset$. By Lemmas \ref{lem:finite-k} (ii) and \ref{lem:colse-1} (ii) it follows that
\[
y_{k,\ell}:=\pi(m 1^{v(k-1)}\;(i_0+1)1((m-1)m^{u-1})^\ell\c)\in\uu_k(\Phi),
\]
where $\pi(\c)\in f_{i_0}(E)\cap \uu_1(\Phi)$. However, the limit
\[
y_k=\lim_{\ell\to\f}y_{k,\ell}=\pi(m 1^{v(k-1)}\;(i_0+1)1((m-1)m^{u-1})^\f)
\]
has a coding ending with $((m-1)m^{u-1})^\f \sim m1^\f$. So by Proposition \ref{cor:countable} it follows that $y_k\in\uu_{\aleph_0}(\Phi)$. Therefore, $\uu_k(\Phi)$ is not closed.
\end{proof}

Now we consider the closeness of $\uu_k(\Phi)$ by assuming $f_1(I)\cap f_2(I)=f_{m-1}(I)\cap f_m(I)=\emptyset$. Note by Lemma \ref{lem:29} that $\uu_k(\Phi)=\emptyset$ if $k\ne 2^\ell$ with $\ell\in\N\cup\set{0}$.
\begin{lemma}\label{lem:close-2}
Let $\Phi=(E, \set{f_i}_{i=1}^m)\in\ee$ with $I=conv(E)$. If $f_1(I)\cap f_2(I)=f_{m-1}(I)\cap f_m(I)=\emptyset$ and $f_i(I)\cap f_{i+1}(I)=f_{im^u}(I)=f_{(i+1)1^v}(I)$ for some $i\in\set{2,\ldots, m-2}$, then
\[
\pi(1^n (im^u)^\ell 1^\f)\in\uu_{2^\ell}(\Phi)
\]
for any integers $n, \ell\ge 0$.
\end{lemma}
\begin{proof}
Since $f_1(I)\cap f_2(I)=\emptyset$, it suffices to prove
\[
x_\ell:=\pi(1(im^u)^\ell 1^\f)\in\uu_{2^\ell}(\Phi).
\]
We prove this by induction on $\ell$. Clearly, for $\ell=0$ we have $x_0=\pi(1^\f)=a\in\uu_1(\Phi)$. Now suppose $x_\ell\in\uu_{2^\ell}(\Phi)$ for some $\ell\ge 0$, and we will prove $x_{\ell+1}\in\uu_{2^{\ell+1}}(\Phi)$.

Since $f_1(I)\cap f_2(I)=\emptyset$, by using $\pi(1(im^u)^\ell 1^\f)=x_\ell\in\uu_{2^\ell}(\Phi)$ it follows that $\pi((im^u)^\ell 1^\f)\in\uu_{2^\ell}(\Phi)$, and  then by using $f_{m-1}(I)\cap f_m(I)=\emptyset$ we have
\begin{equation}\label{eq:close-3}
  \pi(m(im^u)^\ell 1^\f)\in\uu_{2^\ell}(\Phi).
\end{equation}
Let $(d_i)$ be a coding of $x_{\ell+1}=\pi(1(im^u)^{\ell+1} 1^\f)$. Since $f_1(I)\cap f_2(I)=\emptyset$, it follows that $d_1=1$. For the second digit $d_2$ we claim that $d_2\in\set{i, i+1}$.

If $d_2\notin\set{i, i+1}$, then by using   $\pi(d_2d_3\ldots)\in f_i(I)$ we must have $d_2=i-1$. So $\pi((im^u)^{\ell+1} 1^\f)\in f_{i-1}(I)\cap f_i(I)\subset f_{i1}(I)$, which implies
\[
\pi(m^u (im^u)^{\ell} 1^\f)\in f_1(I),
\]
leading to a contradiction with $f_1(I)\cap f_m(I)=\emptyset$.
Therefore,  $d_2=i$ or $d_2=i+1$.

By the claim it follows that
\[
\pi(d_2d_3\ldots)=\pi((im^u)^{\ell+1}1^\f)\in f_i(I)\cap f_{i+1}(I)=f_{im^u}(I)=f_{(i+1)1^v}(I).\]
This together with Lemma \ref{lem:key lem} implies that either $d_2\ldots d_{u+1}=im^{u-1}$ or $d_2\ldots d_{v+1}=(i+1)1^{v-1}$.
Observe that
\[\pi(im^{u-1}\;m(im^u)^\ell 1^\f)=\pi((i+1)1^{v-1}\;1(im^u)^\ell 1^\f),\]
where we have used $f_{im^u}=f_{(i+1)1^v}$ by Lemma \ref{lem:21}. Then by (\ref{eq:close-3}) and the induction hypothesis $\pi(1(im^u)^\ell 1^\f)=x_\ell\in\uu_{2^\ell}(\Phi)$  it follows that $\pi(d_2d_3\ldots)=\pi((im^u)^{\ell+1}1^\f)$ has precisely $2^\ell+2^\ell=2^{\ell+1}$ different codings. So $x_{\ell+1}\in\uu_{2^{\ell+1}}(\Phi)$.
  By induction this completes the proof.
\end{proof}

Given $\de>0$,  for a subset $F$ of $\R$ we define its \emph{$\de$-neighborhood}   by
\[F_\de:=\set{x\in\R: |x-y|< \de~\textrm{for some }y\in F}.\]
In the following proposition we show that $\uu_1(\Phi)$ is closed by assuming $f_1(I)\cap f_2(I)=f_{m-1}(I)\cap f_m(I)=\emptyset$.
\begin{proposition}\label{prop:close-notclose}
Let $\Phi=(E, \set{f_i}_{i=1}^m)\in\ee$ with $I=conv(E)$. If $f_1(I)\cap f_2(I)=f_{m-1}(I)\cap f_m(I)=\emptyset$, then $\uu_k(\Phi)$ is closed if and only if $k\ne 2^\ell$ with $\ell\in\N$. In particular, $\uu_1(\Phi)$ is closed.
\end{proposition}
\begin{proof}
First we prove the necessity. Suppose $k=2^\ell$ for some integer $\ell\ge 1$. By Conditions (C) and (D) there exists $i\in\set{2,\ldots, m-2}$ such that $f_i(I)\cap f_{i+1}(I)=f_{im^u}(I)=f_{(i+1)1^v}(I)$. Then by Lemma \ref{lem:close-2} it follows that
\[
x_{n}:=\pi(1^n(im^u)^\ell 1^\f)\in\uu_{k}(\Phi)\quad\forall n\in\N.
\]
However, the limit $x_\f:=\lim_{n\to\f}x_n=\pi(1^\f)=a\in\uu_1(\Phi)$. So, $\uu_k(\Phi)$ is not closed for any $k=2^\ell$ with $\ell\in\N$.

To prove the sufficiency, in view of Lemma \ref{lem:29}, it suffices to prove that $\uu_1(\Phi)$ is closed. Let $T:\bigcup_{i=1}^{m} f_i(I)\rightarrow I$ be the inverse map of the IFS $\set{f_i}_{i=1}^m$, i.e.,
\[
T(x)=f_i^{-1}(x)\quad\textrm{if}\quad x\in f_i(I).
\]
Then $T$ is a multivalued map satisfying $E=T(E)$. So, if $x$ has a coding $(d_i)\in\set{1,\ldots, m}^\N$, then $T(x)$ has a coding $\si((d_i))=(d_{i+1})$, where $\si$ is the left-shift map. Note that $f_1(I)\cap f_2(I)=f_{m-1}(I)\cap f_m(I)=\emptyset$. Then
\[
g:=\min\set{dist(f_1(I), f_2(I)), dist(f_{m-1}(I), f_m(I))}>0.
\]
By Condition (D) it follows that if $f_i(I)\cap f_{i+1}(I)\ne\emptyset$, then there exist $u_i, v_i\in\N$ such that $f_i(I)\cap f_{i+1}(I)=f_{im^{u_i}}(I)=f_{(i+1)1^{v_i}}(I)$. Let
\[\de:=\min\set{r_i r_m^{u_i}\,g: f_{i}(I)\cap f_{i+1}(I)\ne\emptyset}>0.\]
Now we claim that
\begin{equation}\label{eq:unique-equiv}
\uu_1(\Phi)=\set{x\in E: T^n(x)\notin O_\de~\forall n\ge 0},
\end{equation}
where $O_\de$ is the $\delta$-neighborhood of $O:=\bigcup_{i=1}^{m-1}(f_i(I)\cap f_{i+1}(I))$.

First we prove the inclusion `$\supset$' in (\ref{eq:unique-equiv}). Take $x\notin\uu_1(\Phi)$. Then $x$ has (at least) two different codings $(d_i), (d_i')$ such that
\[
d_1\ldots d_N=d_1'\ldots d_N'\quad\textrm{and}\quad |d_{N+1}-d_{N+1}'|=1
\]
for some integer $N\ge 0$. This implies $T^N(x)\in f_{d_{N+1}}(I)\cap f_{d_{N+1}'}(I)\subset O_\de$, and then establishes the inclusion `$\supset$'.

To prove the inverse inclusion in (\ref{eq:unique-equiv}) it suffices to prove that any $x\in O_\de\cap E$ has at least two different codings. By the definition of $\de$ it follows that
\[O_\de\cap E=O\cap E.\] Take $x\in O\cap E$. Then $x\in f_i(I)\cap f_{i+1}(I)\cap E$ for some $i\in\set{2,\ldots,m-2}$. By Condition (D) there exist $u=u_i, v=v_i\in\N$ such that
\[
f_i(I)\cap f_{i+1}(I)=f_{im^u}(I)=f_{(i+1)1^v}(I),
\]
which implies  $f_{im^u}=f_{(i+1)1^v}$ by Lemma \ref{lem:21}. So by Lemma \ref{lem:key lem} any $x\in f_i(I)\cap f_{i+1}(I)\cap E$ has at least two different codings: one begins with $im^{u-1}$, and the other  begins with $(i+1)1^{v-1}$. This establishes the claim.

Note that $O_\de$ is the finite union of open intervals, and then it is an open set. By (\ref{eq:unique-equiv}) it follows that
\[
\uu_1(\Phi)=E\setminus\bigcup_{n=0}^\f T^{-n}(O_\de),
\]
and then it is closed. This completes the proof.
\end{proof}

\begin{proof}
[Proof of (i) $\Leftrightarrow$ (v) in Theorem \ref{th:1} (A) and (B)]
The equivalence (i) $\Leftrightarrow$ (v) follows by Propositions \ref{prop:notclose} and \ref{prop:close-notclose}.
\end{proof}

\section{Geometrical structure of $E$ and $\uu_1(\Phi)$}\label{sec:geometrical structure}
In this section  we show that the self-similar set $E$ and the set $\uu_1(\Phi)$ of points with a unique coding can be both described as the strongly connected graph-directed set satisfying the OSC  (see \cite{Mauldin_Williams_1988}).

\subsection{Geometrical structure of $\uu_1(\Phi)$}
Given $\Phi=(E, \set{f_i}_{i=1}^m)\in\ee$,  first we consider the graph-directed structure of  the univoque set $\uu_1(\Phi)$. Let
\[
\mathbf F:=\bigcup_{f_i(I)\cap f_{i+1}(I)=f_{im^{u_i}}(I)=f_{(i+1)1^{v_i}}(I)}\set{i m^{u_i}, (i+1)1^{v_i}}
\]
be the set of forbidden blocks, and let
\[
X_{\mathbf F}:=\set{(d_i)\in\set{1,\ldots, m}^\N: (d_i)\textrm{ does not contain any block from }\mathbf F}.
\]
Then $(X_{\mathbf F}, \si)$ is a $N$-step  subshift of finite type  (cf.~\cite{Lind_Marcus_1995}), where $N+1$ is the length of the longest word in $\mathbf F$, and $\si$ is the left shift map. For $n\in\N\cup\set{0}$ we denote by $B_n(X_{\mathbf F})$ the set of all length $n$ admissible words in $X_{\mathbf F}$, i.e.,
\[
B_n(X_{\mathbf F})=\set{\mathbf d=d_1\ldots d_n: \mathbf d\textrm{ occurs in some sequence of }X_{\mathbf F}}.
\]
In particular, for $n=0$ the set $B_0(X_{\mathbf F})$ is the singleton consisting of the empty word $\epsilon$. We will show that $(X_{\mathbf F}, \si)$ is transitive, which means any two admissible words  can be connected in $X_{\mathbf F}$.

Now we construct a directed graph $\mathcal G=(B_{N}(X_{\mathbf F}), \mathbf E)$ for the set $\uu_1(\Phi)$.
Let $B_{N}(X_{\mathbf F})$ be the set of vertices of our graph $\mathcal G$. So each admissible word of length $N$ in $X_{\mathbf F}$ is a vertex of $\mathcal G$.    For two vertices $\mathbf c=c_1\ldots c_{N}, \mathbf d=d_1\ldots d_{N}\in B_{N}(X_{\mathbf F})$ we draw a directed edge from $\mathbf c$ to $\mathbf d$, denoted by $\stackrel{\rightarrow}{\mathbf c \mathbf d}$,  if
\[c_2\ldots c_{N}=d_1\ldots d_{N-1}\quad \textrm{and}\quad c_1\ldots c_{N} d_{N}\in B_{N+1}(X_{\mathbf F}).\]
   In this case the corresponding map for the edge $\stackrel{\rightarrow}{\mathbf c \mathbf d}$ is denoted by  $f_{\stackrel{\rightarrow}{\mathbf c \mathbf d}}=f_{c_1}$. Let $\mathbf E$ be the set of all directed edges in $\mathcal G$. We will show that  the set  $\uu_1(\Phi)$ can be identical to the graph-directed set  satisfying the OSC.

Note that  the set $\uu_1(\Phi)$ is not alway closed (see Proposition \ref{prop:notclose}). We define a countable subset $\mathcal N$ of $\pi(X_{\mathbf F})$ which is not included in $\uu_1(\Phi)$.
\begin{definition}
The subset $\mathcal N$ of $\pi(X_{\mathbf F})$ is defined as follows:
\begin{itemize}
\item If $f_1(I)\cap f_2(I)=f_{21^{v}}(I)$ and $f_{m-1}(I)\cap f_m(I)=\emptyset$, then
\[\mathcal N=\bigcup_{n=0}^\f\bigcup_{\mathbf d\in B_n(X_{\mathbf F})}\set{\pi(\mathbf d(21^{v-1})^\f)};\]
\item If $f_{1}(I)\cap f_2(I)=\emptyset$ and $f_{m-1}(I)\cap f_{m}(I)=f_{(m-1)m^p}(I)$, then
\[\mathcal N=\bigcup_{n=0}^\f\bigcup_{\mathbf d\in B_n(X_{\mathbf F})}\set{\pi(\mathbf d((m-1)m^{p-1})^\f)};\]

\item If $f_1(I)\cap f_2(I)=f_{21^{v}}(I)$ and $f_{m-1}(I)\cap f_m(I)=f_{(m-1)m^p}(I)$, then
\[
\mathcal N=\bigcup_{n=0}^\f\bigcup_{\mathbf d\in B_n(X_{\mathbf F})}\set{\pi(\mathbf d(21^{v-1})^\f), \pi(\mathbf d((m-1)m^{p-1})^\f)}.
\]

\item If $f_1(I)\cap f_2(I)=f_{m-1}(I)\cap f_m(I)=\emptyset$, then $\mathcal N=\emptyset$.
\end{itemize}
\end{definition}
In fact, by Proposition \ref{cor:countable} it follows that any point in $\mathcal N$ has countably many codings.
For each word $\mathbf c=c_1\ldots c_{N}\in B_{N}(X_{\mathbf F})$ we set
\[
\uu_{\mathbf c}:=\set{\pi((d_i)): (d_i)\in X_{\mathbf F}, ~ d_1\ldots d_{N}=\mathbf c},
\]
where $\pi$ is the projection map defined in (\ref{eq:k1}).
By Lemma \ref{lem:21} and our construction of $X_{\mathbf F}$ it follows that
\begin{equation}\label{eq:U-1}
\uu_1(\Phi)\cup\mathcal N=\pi(X_{\mathbf F})=\bigcup_{\mathbf c\in B_{N}(X_{\mathbf F})}\uu_{\mathbf c},
\end{equation}
where the union in the above equation is pairwise disjoint. Furthermore, for each $\mathbf c\in B_{N}(X_{\mathbf F})$ we have
\begin{equation}\label{eq:U-2}
\uu_{\mathbf c}=\bigcup_{\stackrel{\rightarrow}{\mathbf c \mathbf d}\in\mathbf E}f_{\stackrel{\rightarrow}{\mathbf c \mathbf d}}(\uu_{\mathbf d}).
\end{equation}
One can also verify that the union in (\ref{eq:U-2}) is pairwise disjoint. Therefore, by (\ref{eq:U-1}) and (\ref{eq:U-2}) it follows that  up to a countable set $\mathcal N$ the set $\uu_1(\Phi)$ is a graph-directed set satisfying the OSC.

\begin{lemma}\label{lem:transitivity-U}
The subshift of finite type $X_{\mathbf F}$ is transitive. Or equivalently, the directed graph  $\mathcal G$ is strongly connected.
\end{lemma}
\begin{proof}
By our construction it is clear that $X_{\mathbf F}$ is a $N$-step subshift of finite type. So it suffices to prove that $X_{\mathbf F}$ is transitive.

By Condition (C) there exists $j\in\set{1,\ldots, m-1}$ such that
\begin{equation}\label{eq:trans-separation}
f_{j}(I)\cap f_{j+1}(I)=\emptyset.
\end{equation}
For any two admissible word $\mathbf c=c_1\ldots c_p$ and $\mathbf d=d_1\ldots d_q$ in $X_{\mathbf F}$ we will construct a word $\eta$ such that $\mathbf c \eta\mathbf d$ is still an admissible word in $X_{\mathbf F}$. In other words, the word $\mathbf c \eta\mathbf d$ does not contain any block from $\mathbf F$. Note that $X_{\mathbf F}$ is a   $N$-step subshift of finite type. We prove the transitivity in the following two steps. In the first step we show that $\mathbf c$ can be extended to the right which gives an admissible word $\mathbf c w$ ending with $m^{N}$; in the second step we show that the word $m^{N}$ can be connected to $\mathbf d$ via a word $w'$. Then by using \cite[Theorem 2.1.8]{Lind_Marcus_1995} it follows that $\mathbf c w w'\mathbf d$ is an admissible word in $X_{\mathbf F}$.

{\bf Step I.} We extend the word $\mathbf c$ to the right such that the new word $\mathbf c w$ is admissible in $X_{\mathbf F}$ and  ends with $m^{N}$. This will be verified by the following five cases.

(i). $\mathbf c=c_1\ldots c_p$ ends with $im^k$ for some $i\ne m$ and $k\in\set{1,\ldots, p-1}$. If $f_i(I)\cap f_{i+1}(I)=\emptyset$, then we can choose $w=m^{N}$, and one can easily check that $\c w m^\f\in X_{\mathbf F}$. If $f_i(I)\cap f_{i+1}(I)\ne\emptyset$, then there exist $u=u_i, v=v_i\in\N$ such that
\[
f_i(I)\cap f_{i+1}(I)=f_{im^u}(I)=f_{(i+1)1^v}(I).
\]
Since $\mathbf c=c_1\ldots c_{p-k-1}im^k$ is an admissible word, we have $k<u$. By (\ref{eq:trans-separation}) it follows that
\begin{itemize}
\item if $j\ne 1$, then we can take $w=jm^{N}$, and thus $\mathbf c w m^\f=c_1\ldots c_{p-k-1}im^k jm^\f\in X_{\mathbf F}$;
\item if $j=1$, then we can choose $w=(j+1)jm^{N}=21m^{N}$, and thus $\mathbf c w m^\f=c_1\ldots c_{p-k-1}im^k 21 m^\f\in X_{\mathbf F}$.
\end{itemize}

(ii). $\mathbf c=m^p$. Then by a similar argument as in Case (i) one can verify that $\mathbf c m^\f\in X_{\mathbf F}$. So we can take $w=m^N$ in this case.

(iii). $\mathbf c=c_1\ldots c_p$ ends with $i1^\ell$ for some $i\ne 1$ and $\ell\in\set{1,\ldots, p-1}$. If $f_{i-1}(I)\cap f_i(I)=\emptyset$, then in view of (\ref{eq:trans-separation}) we can take
\begin{equation}\label{eq:w}
w=\left\{\begin{array}{lll}
(j+1)jm^{N}&\quad\textrm{if }& j+1\ne m,\\
jm^{N}=(m-1)m^{N}&\quad\textrm{if }& j+1=m.
\end{array}\right.
\end{equation}
Indeed, one can check that $\mathbf c w m^\f=c_1\ldots c_{p-\ell-1}i 1^\ell (j+1) jm^\f\in X_{\mathbf F}$ if $j+1\ne m$, and $\mathbf c w m^\f=c_1\ldots c_{p-\ell-1}i 1^\ell (m-1)m^\f\in X_{\mathbf F}$ if $j+1=m$.

If $f_{i-1}(I)\cap f_i(I)\ne\emptyset$, then there exist $u=u_{i-1}, v=v_{i-1}\in\N$ such that
\[
f_{i-1}(I)\cap f_i(I)=f_{(i-1)m^u}(I)=f_{i1^v}(I).
\]
Since $\mathbf c=c_1\ldots c_{p-\ell-1}i 1^\ell$ is admissible,  we have $\ell<v$. Then we can take
\[
w=\left\{\begin{array}{lll}
jm^{N}&\quad\textrm{if }& j\ne 1,\\
(j+1)jm^{N}=21m^{N}&\quad\textrm{if }& j=1.
\end{array}\right.
\]
One can also verify that $\mathbf c w m^\f\in X_{\mathbf F}$.

(iv). $\mathbf c=1^p$. By a similar argument as in Case (iii) we can take the word $w$ as defined in (\ref{eq:w}). One can check that $\mathbf c w m^\f\in X_{\mathbf F}$.

(v). $\mathbf c=c_1\ldots c_p$  with $c_p=i\notin\set{1, m}$. If $f_i(I)\cap f_{i+1}(I)=\emptyset$, then we can take $w=m^{N}$ since $\mathbf cwm^\f=c_1\ldots c_{p-1}i m^\f\in X_{\mathbf F}$. If $f_i(I)\cap f_{i+1}(I)\ne\emptyset$, then in view of  (\ref{eq:trans-separation}) we can take $w$ as defined in (\ref{eq:w}).
Again, one can verify that $\mathbf c w m^\f\in X_{\mathbf F}$.

{\bf Step II.} There exists a  word $w'$ such that $m^{N}w'\mathbf d$ is admissible in $X_{\mathbf F}$. Suppose $(d_i)\in X_{\mathbf F}$ begins with $\mathbf d$. If $d_1\ne 1$, then take $w'=m$, and one can check that $m^{N}w'd_1d_2\ldots=m^{N+1} d_1d_2\ldots\in X_{\mathbf F}$. If $d_1=1$, then we take $w'=j+1$, and thus by (\ref{eq:trans-separation}) it follows that $m^{N}(j+1)d_1d_2\ldots\in X_{\mathbf F}$.

By Steps I and II it follows that $X_{\mathbf F}$ is a transitive subshift of finite type. So the directed graph $\mathcal G$ is strongly connected.
\end{proof}
By (\ref{eq:U-1}), (\ref{eq:U-2}) and Lemma \ref{lem:transitivity-U} it follows that $\uu_1(\Phi)$ is identical to a strongly connected graph-directed set satisfying the OSC. Then by  using \cite[Theorem 3]{Mauldin_Williams_1988}   the Hausdorff dimension of  $\uu_1(\Phi)$ can  be calculated via the spectrum radius of the corresponding  adjacency matrix of $\mathcal G$, and the corresponding Hausdorff measures of   $\uu_1(\Phi)$ is positive and finite.
\begin{proposition}\label{prop:measure-u1}
Let $\Phi=(E, \set{f_i}_{i=1}^m)\in\ee$.
Then up to a countable set the univoque set $\uu_1(\Phi)$ is a strongly connected graph-directed set satisfying the OSC. So, for $s=\dim_H\uu_1(\Phi)$ the $s$-dimensional Hausdorff measure of $\uu_1(\Phi)$ is positive and finite.
\end{proposition}

\subsection{Geometrical structure of $E$}  Similar to the construction of $\mathcal G$, we construct a directed graph $\mathcal G_*$ for the self-similar set $E$. Let
\[
\mathbf F_*:=\set{im^{u_i}: f_i(I)\cap f_{i+1}(I)=f_{im^{u_i}}(I)=f_{(i+1)1^{v_i}}(I)}
\]
be the set of forbidden blocks, and let
\[
X_{\mathbf F_*}:=\set{(d_i)\in\set{1,\ldots, m}^\N: (d_i) \textrm{ does not conatin any word from }\mathbf F_*}.
\]
Then $(X_{\mathbf F_*}, \si)$ is a $N_*$-step subshift of finite type, where $N_*+1$ is the length of the longest word in $\mathbf F_*$.

Now we describe the graph-directed structure of $E$ based on the $N_*$-step subshift of finite type $X_{\mathbf F_*}$.
We construct a directed graph $\mathcal G_*=(B_{N_*}(X_{\mathbf F_*}), \mathbf E_*)$ in the following way. Let $B_{N_*}(X_{\mathbf F_*})$ be the set of vertices of our graph $\mathcal G_*$. For two vertices   $\mathbf c=c_1\ldots c_{N_*}$ and $\mathbf d=d_1\ldots d_{N_*}\in B_{N_*}(X_{\mathbf F_*})$ we  connect a directed edge from $\mathbf c$ to $\mathbf d$, denoted by $\stackrel{\rightarrow}{\mathbf c \mathbf d}$,  if $c_2\ldots c_{N_*}=d_1\ldots d_{N_*-1}$ and $c_1\ldots c_{N_*} d_{N_*}\in B_{N_*+1}(X_{\mathbf F_*})$.  In this case we write the corresponding map  $f_{\stackrel{\rightarrow}{\mathbf c \mathbf d}}=f_{c_1}$.   Let $\mathbf E_*$ be the collection of all directed edges in the graph $\mathcal G_*$.   For each word $\mathbf c=c_1\ldots c_{N_*}\in B_{N_*}(X_{\mathbf F_*})$ let
\[E_{\mathbf c}:=\set{\pi((d_i)): (d_i)\in X_{\mathbf F_*}, ~d_1\ldots d_{N_*}=c_1\ldots c_{N_*}}.\]
 Then the self-similar set $E$ can be written as
\begin{equation}\label{eq:E-1}
E=\pi(X_{\mathbf F_*})=\bigcup_{\mathbf c\in B_{N_*}(X_{\mathbf F_*})}E_{\mathbf c}.
\end{equation}
By Lemma \ref{lem:21} and our construction of $X_{\mathbf F_*}$ it follows that the union in (\ref{eq:E-1}) is pairwise disjoint.  Furthermore, for each $\mathbf c\in B_{N_*}(X_{\mathbf F_*})$ we have
\begin{equation}\label{eq:E-2}
E_{\mathbf c}=\bigcup_{\stackrel{\rightarrow}{\mathbf c \mathbf d}\in\mathbf E_*} f_{\stackrel{\rightarrow}{\mathbf c \mathbf d}}(E_{\mathbf d}).
\end{equation}
One can also verify that the union in   (\ref{eq:E-2}) is pairwise disjoint.  Therefore, by (\ref{eq:E-1}) and (\ref{eq:E-2}) it follows that $E$ is a graph-directed set  satisfying the OSC which can be represented by the directed graph $\mathcal G_*$.

By a similar way as in the proof of Lemma \ref{lem:transitivity-U} one can show that the graph $\mathcal G_*$ is strongly connected.
\begin{lemma}\label{lem:transitivity-E}
The subshift of finite type $(X_{\mathbf F_*}, \si)$ is transitive. Or equivalently, the graph $\mathcal G_*$ is strongly connected.
\end{lemma}

Hence, by (\ref{eq:E-1}), (\ref{eq:E-2}), Lemma \ref{lem:transitivity-E} and using \cite[Theorem 3]{Mauldin_Williams_1988} it follows that the Hausdorff dimension of $E$   can   be calculated via the spectrum radius of the  adjacency matrix of $\mathcal G_*$, and then the corresponding Hausdorff measure of $E$ is   positive and finite.
\begin{proposition}\label{prop:measure-E}
Let $\Phi=(E, \set{f_i}_{i=1}^m)\in\ee$.
Then  the  set $E$ is a  strongly connected graph-directed set satisfying the OSC. So, for $t=\dim_H E$ the $t$-dimensional Hausdorff measure of $E$ is positive and finite.
\end{proposition}

In the next section we will explicitly determine the Hausdorff dimensions of $\uu_1(\Phi)$ and $E$.

\section{Hausdorff dimensions and Hausdorff measures of $E$ and  $\uu_k(\Phi)$}\label{sec:hausdorff measure-uk}
Let $\Phi=(E, \set{f_i}_{i=1}^m)\in\ee$. In this section we will investigate the Hausdorff dimensions and Hausdorff measures of $E$ and $\uu_k(\Phi)$, and prove Theorem \ref{th:measures}.

\subsection{Hausdorff dimensions of $E$ and $\uu_k(\Phi)$}
Recall from Propositions \ref{prop:measure-u1} and \ref{prop:measure-E}  that both the self-similar set $E$ and the univoque set $\uu_1(\Phi)$ can be identical to  strongly connected  graph-directed sets satisfying the  OSC, and thus the corresponding Hausdorff measures of $E$ and $\uu_1(\Phi)$ are positive and finite. Based on this and by using Bonferroni inequality we are able to determine the explicit formulae for the Hausdorff dimensions of $E$ and $\uu_k(\Phi)$, respectively. Moreover, we prove that the   Hausdorff measure of $\uu_k(\Phi)$ is infinite for any $k\ge 2$ satisfying $\uu_k(\Phi)\ne\emptyset$.

Given $\Phi=(E, \set{f_i}_{i=1}^n)\in\ee$, by Propositions \ref{prop:measure-u1} and \ref{prop:measure-E} it follows that the corresponding Hausdorff measures of $E$ and $\uu_1(\Phi)$ are both positive and finite, i.e., for $s:=\dim_H \uu_1(\Phi)$ and $t:=\dim_H E$ we have
\begin{equation}\label{eq:measure-positive-finite}
\mathcal H^s(\uu_1(\Phi))\in(0, \f)\quad\textrm{and}\quad \mathcal H^t(E)\in(0, \f).
\end{equation}
Furthermore, note that $\uu_1(\Phi)$ is a proper subset of $E$, and $E$ is a strongly connected graph-directed set satisfying the OSC. Then $\dim_H\uu_1(\Phi)<\dim_H E$. Note that
\[
E=\uu_{2^{\aleph_0}}(\Phi)\cup \uu_{\aleph_0}(\Phi)\cup\bigcup_{k=1}^\f\uu_k(\Phi).
\]
So, by (\ref{eq:measure-positive-finite}) and Theorem \ref{th:1} it follows that
\begin{equation}\label{eq:uncountable}
\dim_H\uu_{2^{\aleph_0}}(\Phi)=\dim_H E,\quad\textrm{and}\quad \mathcal H^t(\uu_{2^{\aleph_0}}(\Phi))=\mathcal H^t(E)\in(0, \f).
\end{equation}
In the remaining part of this subsection we will focus on the explicit formulae for the Hausdorff dimensions of $E$ and $\uu_1(\Phi)$, respectively. This will be done by the Bonferroni inequality (cf.~\cite[Exercise 3.12]{Durrett-2010}).

For $\Phi=(E, \set{f_i}_{i=1}^m)\in\ee$ we recall from Definition \ref{def:overlap-vector} that the  {overlapping   vectors} $\mathbf u=(u_1,\ldots, u_m), \mathbf v=(v_1,\ldots, v_m)$ of $\Phi$ are defined as follows: if $f_i(I)\cap f_{i+1}(I)\ne\emptyset$, then there exist $u, v\in\N$ such that $f_i(I)\cap f_{i+1}(I)=f_{im^u}(I)=f_{(i+1)1^v}(I)$, and in this case $u_i=u$ and $v_i=v$. Otherwise,  $u_i=v_i=\f$.
In particular, $u_m=v_m=\f$.
\begin{proposition}\label{lem:dim-E}
Let $\Phi=(E, \set{f_i}_{i=1}^m)\in\ee$ with $f_i(x)=r_i x+b_i$ for $1\le i\le m$. Then the Hausdorff dimension  $t=\dim_H E$ satisfies
\[
\sum_{i=1}^m r_i^t(1-r_m^{u_i t})=1.
\]
\end{proposition}
\begin{proof}
Note that $E=\bigcup_{i=1}^m f_i(E)$, and by Condition (B) that any three basic intervals have empty intersection. Then by Bonferroni inequality it follows that
\begin{equation}\label{eq:bonferroni}
\mathcal H^t(E)=\mathcal H^t\left(\bigcup_{i=1}^m f_i(E)\right)=\sum_{i=1}\mathcal H^t(f_i(E))-\sum_{i=1}^{m-1}\mathcal H^t(f_i(E)\cap f_{i+1}(E)),
\end{equation}
where $t=\dim_H E$.
Observe by Condition (D) that if $f_i(I)\cap f_{i+1}(I)\ne\emptyset$, then $f_i(I)\cap f_{i+1}(I)=f_{i m^{u_i}}(I)$ for some positive integer $u_i$. By Lemma \ref{lem:21} this also implies
\[f_i(E)\cap f_{i+1}(E)=f_{im^{u_i}}(E).\]
 Since the Hausdorff measure $\mathcal H^t$ is translation invariant, using the scaling property  of the Hausdorff measure in (\ref{eq:bonferroni}) it follows that
\begin{align*}
\mathcal H^t(E)&=\sum_{i=1}^m  \mathcal H^t(f_i(E))-\sum_{f_{i}(I)\cap f_{i+1}(I)\ne\emptyset}\mathcal H^t(f_{im^{u_i}}(E))\\
&=\sum_{i=1}^m r_i^t\mathcal H^t(E)-\sum_{f_{i}(I)\cap f_{i+1}(I)\ne\emptyset}r_i^t r_m^{u_i t}\mathcal H^t(E)\\
&=\sum_{i=1}^{m} r_i^t(1-r_m^{u_i t}) \mathcal H^t(E),
\end{align*}
where the last equality holds by using that $r_m^{u_i t}=0$ if $f_i(I)\cap f_{i+1}(I)=\emptyset$. Note by (\ref{eq:measure-positive-finite}) that $\mathcal H^t(E)\in(0, \f)$. Then the lemma follows by dividing $\mathcal H^t(E)$ on both sides of the above equation.
\end{proof}
Now we turn to determine the explicit formula for the Hausdorff dimension of $\uu_1(\Phi)$.
First we need the following   lemma.  Set $\uu:=\uu_1(\Phi)$, and for a word $\mathbf w\in\set{1,\ldots, m}^*$ write $\uu(\mathbf w):=\uu\cap f_{\mathbf w}(E)$. Then any $x\in\uu(\mathbf w)$ has a unique coding with a prefix $\mathbf w$.

\begin{lemma}\label{lem:dim-u0}
Let $\Phi=(E, \set{f_i}_{i=1}^m)\in\ee$ with $f_i(x)=r_i x+b_i$ for $1\le i\le m$.
\begin{enumerate}[{\rm(i)}]
\item For any $1\le i\le m$,
\[
\uu(i)=\left\{
\begin{array}{lll}
f_i(\uu)&\textrm{if}&  f_{i-1}(I)\cap f_i(I)=f_{i}(I)\cap f_{i+1}(I)=\emptyset,\\
f_i(\uu)\setminus f_i(\uu(1^{v_{i-1}}))&\textrm{if}&  f_{i-1}(I)\cap f_i(I)\ne\emptyset\textrm{ and }f_{i}(I)\cap f_{i+1}(I)=\emptyset,\\
f_i(\uu)\setminus f_i(\uu(m^{u_i}))&\textrm{if}&  f_{i-1}(I)\cap f_i(I)=\emptyset\textrm{ and }f_{i}(I)\cap f_{i+1}(I)\ne\emptyset,\\
f_i(\uu)\setminus\big(f_i(\uu(1^{v_{i-1}}))\cup f_i(\uu(m^{u_i}))\big)&\textrm{if}& f_{i-1}(I)\cap f_i(I)\ne\emptyset\textrm{ and }f_{i}(I)\cap f_{i+1}(I)\ne\emptyset.
\end{array}\right.
\]
\item  For any $u, v\in\N$ we have
\[
\uu(m^{u})=\left\{\begin{array}{lll}
f_{m^{u}}(\uu)&\textrm{if}& f_{m-1}(I)\cap f_m(I)=\emptyset,\\
f_{m^{u}}(\uu)\setminus f_{m^{u}}(\uu(1^{v_{m-1}}))&\textrm{if}& f_{m-1}(I)\cap f_m(I)\ne\emptyset,
\end{array}
\right.
\]
and
\[
\uu(1^{v})=\left\{\begin{array}{lll}
f_{1^{v}}(\uu)&\textrm{if}& f_1(I)\cap f_2(I)=\emptyset,\\
f_{1^{v}}(\uu)\setminus f_{1^{v}}(\uu(m^{u_1}))&\textrm{if}& f_1(I)\cap f_2(I)\ne\emptyset.
\end{array}
\right.
\]
\end{enumerate}
\end{lemma}
\begin{proof}
Let $i\in\set{1,2,\ldots, m}$. Since the proofs for different cases in (i) are similar, we only prove  that
\[
\uu(i)=f_i(\uu)\setminus\big( f_i(\uu(1^{v_{i-1}}))\cup f_i(\uu(m^{u_i}))\big)
\]
assuming that  $f_{i-1}(I)\cap f_i(I)\ne\emptyset$ and $f_{i}(I)\cap f_{i+1}(I)\ne\emptyset$.

 Take $x\in\uu(i)$. Then $x$ has a unique coding $\pi^{-1}(x)=x_1x_2\ldots$ with $x_1=i$. Clearly, $\pi(x_2x_3\ldots)\in\uu$. This implies $x=\pi(i x_2x_3\ldots)\in f_i(\uu)$. Since by Condition (D) that $f_{i-1}(I)\cap f_i(I)=f_{(i-1)m^{u_{i-1}}}(I)=f_{i1^{v_{i-1}}}(I)$, any point in $f_i(\uu(1^{v_{i-1}}))$ has at least two different codings: one begins with $i 1^{v_{i-1}}$ and another begins with $(i-1)m^{u_{i-1}}$. This gives $x\notin f_i(\uu(1^{v_{i-1}}))$. Similarly, note by Condition (D) that $f_i(I)\cap f_{i+1}(I)=f_{im^{u_i}}(I)=f_{(i+1)1^{v_i}}(I)$. Then any point  in $f_i(\uu(m^{u_i}))$ has at least two codings: one begins with $i m^{u_i}$ and another begins with $(i+1)1^{v_i}$. So, $x\notin f_i(\uu(m^{u_i}))$. This proves $\uu(i)\subset f_i(\uu)\setminus\big(f_i(\uu(1^{v_{i-1}}))\cup f_i(\uu(m^{u_i}))\big)$.

 To prove the inverse inclusion we take $y\in f_i(\uu)\setminus\big(f_i(\uu(1^{v_{i-1}}))\cup f_i(\uu(m^{u_i}))\big)$. Then $y$ has a coding $y_1y_2\ldots$ satisfies
 \begin{equation}\label{eq:ecnu-1}
 y_1=i\quad\textrm{and}\quad \pi(y_2y_3\ldots)\in\uu.
 \end{equation}
  Furthermore,
  \begin{equation}\label{eq:ecnu-2}
  y_2\ldots y_{1+v_{i-1}}\ne 1^{v_{i-1}}\quad \textrm{and}\quad y_2\ldots y_{1+u_i}\ne m^{u_i}.
  \end{equation}
   It remains to prove $y\in\uu$. Suppose on the contrary that $y$ has another coding $y_1'y_2'\ldots$. If $y_1'=y_1=i$, then
 \[
 \pi(y_2'y_3'\ldots)=\pi(y_2y_3\ldots)
 \]
 has at least two different codings, leading to a contradiction with (\ref{eq:ecnu-1}). So $y_1'\ne y_1$. Note that $y_1=i$. This implies  $y\in f_{i-1}(I)\cap f_i(I)$ or $y\in f_i(I)\cap f_{i+1}(I)$. If $y\in f_{i-1}(I)\cap f_i(I)$, then by   (\ref{eq:ecnu-1}) it follows that $y_2\ldots y_{1+v_{i-1}}=1^{v_{i-1}}$, leading to a contradiction with (\ref{eq:ecnu-2}). If $y\in f_{i}(I)\cap f_{i+1}(I)$, then by (\ref{eq:ecnu-1}) we obtain  $y_2\ldots y_{1+u_i}=m^{u_i}$, again leading to a contradiction with (\ref{eq:ecnu-2}). Therefore, $y_1y_2\ldots$ is the unique coding of $y$, i.e., $y\in\uu(i)$. This completes the proof.

 (ii) can be proved analogously as in (i).
\end{proof}

\begin{proposition}\label{lem:dim-u1}
Let $\Phi=(E, \set{f_i}_{i=1}^m)\in\ee$ with $f_i(x)=r_i x+b_i$ for $1\le i\le m$. Then the Haudorff dimension $s=\dim_H \uu_1(\Phi)$ satisfies
\[
\sum_{i=1}^m r_i^s\left(1- \frac{r_m^{u_i s}(2-r_1^{v_{m-1}s}-r_m^{u_1 s})}{1-r_1^{v_{m-1}s}r_m^{u_1 s}}\right)=1.
\]
\end{proposition}
\begin{proof}
Suppose without loss of generality that $f_1(I)\cap f_2(I)\ne \emptyset$ and $f_{m-1}(I)\cap f_m(I)\ne\emptyset$. Let $\uu=\uu(\Phi)$, and for a word $\mathbf w\in\set{1,2,\ldots, m}^*$ we write $\uu(\mathbf w)=\uu\cap f_{\mathbf w}(E)$. Note that
\[\uu=\bigcup_{i=1}^m\uu(i),\]
 where  the union is pairwise disjoint. By Lemma \ref{lem:dim-u0} (i) it follows that if $f_{i-1}(I)\cap f_i(I)\ne\emptyset$ and  $f_i(I)\cap f_{i+1}(I)\ne\emptyset$, then   $\uu(i)$ can be written as
\[
\uu(i)=f_i(\uu)\setminus\big(f_i(\uu(1^{v_{i-1}})\cup f_i(\uu(m^{u_i})))\big),
\]
where the union is disjoint. For other cases $\uu(i)$ can be written analogously, see Lemma \ref{lem:dim-u0} (i).
Let $s:=\dim_H\uu$. Then the $s$-dimensional Hausdorff measure of $\uu$ can be written as
\begin{equation}\label{eq:ecnu-11}
\begin{split}
\mathcal H^s(\uu)=\sum_{i=1}^m\mathcal H^s(\uu(i))&=\sum_{i=1}^m\Big(\mathcal H^s(f_i(\uu))-\mathcal H^s\big(f_i(\uu(1^{v_{i-1}}))\big)-\mathcal H^s\big(f_i(\uu(m^{u_i}))\big)\Big)\\
&=\sum_{i=1}^m r_i^s\Big(\mathcal H^s(\uu)-\mathcal H^s(\uu(1^{v_{i-1}}))-\mathcal H^s(\uu(m^{u_i}))\Big),
\end{split}
\end{equation}
where we set $v_0:=\f$.
We emphasize that  if $f_{i-1}(I)\cap f_i(I)=\emptyset$, then $v_{i-1}=\f$ which implies  $\mathcal H^s(\uu(1^{v_{i-1}}))=0$. Similarly, if $f_i(I)\cap f_{i+1}(I)=\emptyset$, then $u_i=\f$ which gives $\mathcal H^s(\uu(m^{u_i}))=0$. So (\ref{eq:ecnu-11}) holds for all cases independent of the locations of the basic intervals $f_{i-1}(I), f_i(I)$ and $f_{i+1}(I)$.

Since $f_1(I)\cap f_2(I)\ne\emptyset$ and  $f_{m-1}(I)\cap f_m(I)\ne\emptyset$,  by   Lemma \ref{lem:dim-u0} (ii) it follows that
\begin{align*}
\mathcal H^s(\uu(1^{v_{i-1}}))&=\mathcal H^s(f_{1^{v_{i-1}}}(\uu))-\mathcal H^s\big(f_{1^{v_{i-1}}}(\uu(m^{u_1}))\big)\\
&=r_1^{v_{i-1}s}\Big(\mathcal H^s(\uu)-\mathcal H^s(\uu(m^{u_1}))\Big)\\
&=r_1^{v_{i-1}s}\mathcal H^s(\uu)-r_1^{v_{i-1}s}\Big(\mathcal H^s(f_{m^{u_1}}(\uu))-\mathcal H^s(f_{m^{u_1}}(\uu(1^{v_{m-1}})))\Big)\\
&=r_1^{v_{i-1}s}\mathcal H^s(\uu)-r_1^{v_{i-1}s}r_m^{u_1 s}\mathcal H^s(\uu)+r_1^{v_{i-1}s}r_m^{u_1 s}\mathcal H^s(\uu(1^{v_{m-1}}))\\
&=\mathcal H^s(\uu)r_1^{v_{i-1}s}(1-r_m^{u_1 s})+r_1^{v_{i-1}s}r_m^{u_1 s}\mathcal H^s(\uu(1^{v_{m-1}})).
\end{align*}
Repeating using Lemma \ref{lem:dim-u0} (ii) in the above equation we can deduce that for any $N\in\N\cup\set{0}$,
\begin{align*}
\mathcal H^s(\uu(1^{v_{i-1}}))&=\mathcal H^s(\uu)r_1^{v_{i-1}s}(1-r_m^{u_1 s})\sum_{k=0}^N(r_1^{v_{m-1}s}r_m^{u_1 s})^k\\
&\qquad+r_1^{v_{i-1}s}r_m^{u_1 s}(r_1^{v_{m-1}s}r_m^{u_1 s})^N\mathcal H^s(\uu(1^{v_{m-1}})).
\end{align*}
Letting $N\to\f$, and then $(r_1^{v_{m-1}s}r_m^{u_1 s})^N\to 0$, using that $\mathcal H^s(\uu(1^{v_{m-1}}))\le \mathcal H^s(\uu)<\f$ we obtain
\begin{equation}\label{eq:ecnu-12}
\mathcal H^s(\uu(1^{v_{i-1}}))=\mathcal H^s(\uu)\frac{r_1^{v_{i-1}s}(1-r_m^{u_1 s})}{1-r_1^{v_{m-1}s}r_m^{u_1 s}}.
\end{equation}
Similarly, one can prove
\begin{equation}\label{eq:ecnu-13}
\mathcal H^s(\uu(m^{u_i}))=\mathcal H^s(\uu)\frac{r_m^{u_i s}(1-r_1^{v_{m-1}s})}{1-r_1^{v_{m-1}s}r_m^{u_1 s}}.
\end{equation}

Substituting (\ref{eq:ecnu-12}) and (\ref{eq:ecnu-13}) into (\ref{eq:ecnu-11}), and using that $\mathcal H^s(\uu)\in(0, \f)$ it follows that
\begin{align*}
1&=\sum_{i=1}^m r_i^s-\frac{1-r_m^{u_1 s}}{1-r_1^{v_{m-1}s}r_m^{u_1 s}}\sum_{i=1}^m r_i^s r_1^{v_{i-1}s}-\frac{1-r_1^{v_{m-1}s}}{1-r_1^{v_{m-1}s}r_m^{u_1 s}}\sum_{i=1}^m r_i^s r_m^{u_i s}.
\end{align*}
Note that $v_0=\f, u_m=\f$, and by $f_{i1^{v_{i-1}}}=f_{(i-1)m^{u_{i-1}}}$ that $r_i r_1^{v_{i-1}}=r_{i-1}r_m^{u_{i-1}}$. Rearranging the second summation  in the above equation we conclude that the Hausdorff dimension $s=\dim_H\uu$ satisfies
\begin{align*}
1&=\sum_{i=1}^m r_i^s-\frac{1-r_m^{u_1 s}}{1-r_1^{v_{m-1}s}r_m^{u_1 s}}\sum_{i=1}^m r_i^s r_m^{u_{i}s}-\frac{1-r_1^{v_{m-1}s}}{1-r_1^{v_{m-1}s}r_m^{u_1 s}}\sum_{i=1}^m r_i^s r_m^{u_i s}\\
&=\sum_{i=1}^m r_i^s\left(1-\frac{r_m^{u_i s}(2-r_1^{v_{m-1}s}-r_m^{u_1 s})}{1-r_1^{v_{m-1}s}r_m^{u_1 s}}\right).
\end{align*}
This completes the proof.
\end{proof}

\subsection{Hausdorff measure of $\uu_k(\Phi)$}
Given $\Phi=(E, \set{f_i}_{i=1}^m)\in\ee$, in this subsection we will show that the corresponding Hausdorff measure of $\uu_k(\Phi)$ is infinite for any $k\ge 2$ satisfying $\uu_k(\Phi)\ne\emptyset$.

Let $\Phi=(E, \set{f_i}_{i=1}^m)\in\ee$.   By Lemmas \ref{lem:finite-k} and \ref{lem:29} we construct a large subset of $\uu_k(\Phi)$ as described in the following lemma.
For simplicity we write $\uu:=\uu_1(\Phi)$ and $\us:=\pi^{-1}(\uu)$. Furthermore, we denote by $B_n(\us)$ the set of all length $n$ prefixes of sequences from $\us$.

\begin{lemma}\label{lem:subset of Uk}
Let $\Phi=(E, \set{f_i}_{i=1}^m)\in\ee$ with $I=conv(E)$, and let $k\in\N$.
\begin{enumerate}[{\rm(i)}]
\item If $f_1(I)\cap f_2(I)=f_{1m^u}(I)$ and $f_{i}(I)\cap f_{i+1}(I)=\emptyset$ for some $i\in\set{2,\ldots, m-1}$, then
\[
\pi(c_1\ldots c_n\, 1 m^{u(k-1)}\,\mathbf d)\in\uu_k(\Phi)
\]
for any $c_1\ldots c_n\in B_n(\us)$ with $c_n=i+1$ and for any $\mathbf d=d_1d_2\ldots\in\us$ with $d_1=i$.

\item If $f_{m-1}(I)\cap f_m(I)=f_{m1^v}(I)$ and $f_i(I)\cap f_{i+1}(I)=\emptyset$ for some $i\in\set{1,\ldots, m-2}$, then
\[
\pi(c_1\ldots c_n\,m1^{v(k-1)}\,\mathbf d)\in\uu_k(\Phi)
\]
for any $c_1\ldots c_n\in B_n(\us)$ with $c_n=i$ and for any $\mathbf d=d_1d_2\ldots\in\us$ with $d_1=i+1$.

\item If $f_1(I)\cap f_2(I)=f_{m-1}(I)\cap f_m(I)=\emptyset$ and $f_i(I)\cap f_{i+1}(I)=f_{im^u}(I)$ for some $i\in\set{2,\ldots, m-2}$, then
\[
\pi(c_1\ldots c_n\,(im^u)^{k}\,\mathbf d)\in\uu_{2^{k}}(\Phi)
\]
for any $c_1\ldots c_n\in B_n(\us)$ with $c_n=m$ and for any $\mathbf d=d_1d_2\ldots\in\us$ with $d_1=m$.
\end{enumerate}
\end{lemma}
For $\Phi=(E, \set{f_i}_{i=1}^m)\in\ee$ note by Proposition \ref{prop:measure-u1} that $\uu_1(\Phi)$ is identical to a strongly connected graph-directed set satisfying the OSC.  In \cite{Mauldin_Williams_1988} Mauldin and Williams showed that the Hausdorff dimension $s=\dim_H\uu_1(\Phi)$ can be calculated via the spectral radius of the corresponding  adjacency matrix   $A(s)=(a_{i, j}(s))$ which defined in  the following way.
 Recall from Section \ref{sec:geometrical structure} that $\uu_1(\Phi)$ can be represented by the directed graph $\mathcal G=(B_N(X_{\mathbf F}), \mathbf E)$. The size of the matrix $A(s)$ is $|B_N(X_{\mathbf F})|\times |B_N(X_{\mathbf F})|$. For two vertices $\mathbf c=c_1\ldots c_N$ and $\mathbf d=d_1\ldots d_N$, if $\mathbf c$ is connected to $\mathbf d$, then  we define the map for the edge $\stackrel{\rightarrow}{\mathbf c\mathbf d}$ by$f_{\stackrel{\rightarrow}{\mathbf c\mathbf d}}(x)= f_{c_1}(x)$. In this case, the corresponding entry of $A(s)$ is defined by
\[
a_{\mathbf c, \mathbf d}(s)=r_{c_1}^s.
\]
If $\mathbf c$ is not connected to $\mathbf d$, then we define $a_{\mathbf c,\mathbf d}(s)=0$.

Note by Lemma \ref{lem:transitivity-U} that the directed graph $\mathcal G$ is strongly connected. Then $A(s)$ is an irreducible non-negative   matrix. So it has a unique largest non-negative eigenvalue $\rho(s)$, which is also called the \emph{Perron eigenvalue} of $A(s)$.
It was shown in \cite[Theorem 3]{Mauldin_Williams_1988} that the Hausdorff dimension $s=\dim_H\uu_1(\Phi)$ satisfies
\[\rho(s)=1.\]
\begin{proposition}
\label{prop:measure-Uk}
Let $\Phi=(E, \set{f_i}_{i=1}^m)\in\ee$ and let $s=\dim_H\uu_1(\Phi)$. Then
\[
\mathcal H^s(\uu_k(\Phi))=\f
\]for any $k\ge 2$ satisfying $\uu_k(\Phi)\ne\emptyset$.
\end{proposition}
\begin{proof}
Let $I=conv(E)$ and let $k\ge 2$ with $\uu_k(\Phi)\ne\emptyset$. Since the proof for different cases are similar, we assume without loss of generality that $f_1(I)\cap f_2(I)\ne\emptyset$ and $f_i(I)\cap f_{i+1}(I)=\emptyset$ for some $2\le i<m$. So there exists a positive integer $u$ such that $f_1(I)\cap f_2(I)=f_{1m^u}(I)$.
By Lemma \ref{lem:subset of Uk} (i) it follows that
\[
\bigcup_{n=1}^\f\bigcup_{c_1\ldots c_n\in B_n(\us), c_n=i+1}\set{\pi(c_1\ldots c_n 1 m^{u(k-1)}d_1d_2\ldots): (d_i)\in\us, d_1=i}\subset\uu_k(\Phi),
\]
where $\us=\pi^{-1}(\uu_1(\Phi))$. Furthermore, note that $k\ge 2$, $c_1\ldots c_n$ is an admissible word in $\us$, and $f_{i}(I)\cap f_{i+1}(I)=\emptyset$. Then  one can verify that the union in the above equation is pairwise disjoint. Let $s=\dim_H\uu_1(\Phi)$.  Therefore,
\begin{equation}\label{eq:cq-1}
\begin{split}
\mathcal H^s(\uu_k(\Phi))&\ge \sum_{n=1}^\f\sum_{c_1\ldots c_n\in B_n(\uu), c_n=i+1}\mathcal H^s\Big(\set{\pi(c_1\ldots c_n 1 m^{u(k-1)}d_1d_2\ldots): (d_i)\in\us, d_1=i}\Big)\\
&=D \sum_{n=1}^\f\sum_{c_1\ldots c_n\in B_n(\us), c_n=i+1}\left(\prod_{j=1}^n r_{c_j}^s\right),
\end{split}
\end{equation}
where
\begin{align*}
D&:=\mathcal H^s\big(\set{\pi(1m^{u(k-1)}d_1d_2\ldots): (d_i)\in\mathbf U, d_1=i}\big)\\
&=
r_1^s r_m^{u(k-1)s}\mathcal H^s(f_i(E)\cap \uu_1(\Phi))>0
\end{align*}
using Proposition \ref{prop:measure-u1}.  Note that $D$ is a constant independent of the summation in (\ref{eq:cq-1}). By using the Perron-Frobenius Theorem (cf.~\cite[Chapter 4]{Lind_Marcus_1995}) it follows that
\begin{equation}\label{eq:cq-2}
\sum_{c_1\ldots c_n\in B_n(\us), c_n=i+1}\left(\prod_{j=1}^n r_{c_j}^s\right)\ge D_0\cdot \rho(s)^n
\end{equation}
for some constant $D_0>0$, where $\rho(s)$ is the Perron eigenvalue of the  matrix $\mathbf A(s)$. Note that $\rho(s)=1$. By (\ref{eq:cq-1}) and (\ref{eq:cq-2}) we conclude that
\[
\mathcal H^s(\uu_k(\Phi))\ge \sum_{n=1}^\f D\cdot D_0=\f.
\]
\end{proof}
\begin{proof}[Proof of Theorem \ref{th:measures}]
The Hausdorff dimensions and Hausdorff measures of $E$ and $\uu_{2^{\aleph_0}}(\Phi)$ follow by Propositions \ref{prop:measure-E} and   \ref{lem:dim-E}. This proves (i). For the Hausdorff dimension of  $\uu_k(\Phi)$ it can be deduced from Theorem \ref{th:1} and Proposition \ref{lem:dim-u1}. For the Hausdorff measures of $\uu_k(\Phi)$ it follows from Propositions \ref{prop:measure-u1} and \ref{prop:measure-Uk}. This completes the proof.
\end{proof}

\section{Local dimension of self-similar measure }\label{sec:local dimension} Let $\Phi=(E, \set{f_i}_{i=1}^m)\in\ee$.
Given a probability vector $\p=(p_1,\ldots, p_m)$ with each $p_i>0$, recall from the first section that $\mu_\p$ is the self-similar measure supported on $E$ satisfying
\[
\mu_\p=\sum_{i=1}^m p_i\mu_\p\circ f_i^{-1}.
\]
In this section we will determine the local dimension of $\mu_\p$ at points in $\uu_k(\Phi)$ and $\uu_{\aleph_0}(\Phi)$.

\subsection{Local dimension of $\mu_\p$ at points in $\uu_k(\Phi)$} First we consider the local dimension of $\mu_\p$ at points in $\uu_k(\Phi)$.   Recall that  $B(x, r)=(x-r, x+r)$ is the open interval with center at $x$ and radius $r$.
 \begin{lemma}\label{lem:scaling property}
 Let $h(x)=\alpha x+b$ with $\alpha>0$. Then
 \[
 h(B(c, r))=B(h(c), \alpha r).
 \]
 \end{lemma}
 \begin{proof}
 For any $z=h(y)\in h(B(c, r))$ with $|y-c|< r$ we have
 \[
 |z-h(c)|=\alpha|y-c|<\alpha r,
 \]
 which implies that $z\in B(h(c), \alpha r)$. On the other hand, for any $z\in B(h(c), \alpha r)$ we have
 \[
 \alpha\left|\frac{z-b}{\alpha}-c\right|=\left|h\Big(\frac{z-b}{\alpha}\Big)-h(c)\right|=|z-h(c)|<\alpha r.
 \]
 Thus, $z=h(\frac{z-b}{\alpha})$ with $\frac{z-b}{\alpha}\in B(c, r)$.
 \end{proof}
 In the following proposition we show that the local dimension of $\mu_\p$ at each point in $\uu_k(\Phi)$ is  uniquely  determined by a  point in $\uu_1(\Phi)$.
\begin{proposition}
\label{prop:finite-unique}
Let $\Phi=(E, \set{f_i}_{i=1}^m)\in\ee$ and $k\ge 2$. Then for any $x\in\uu_k(\Phi)$ there exists  a word $\mathbf i\in\set{1,\ldots, m}^*$ and a unique $y\in\uu_1(\Phi)$ such that $x=f_{\mathbf i}(y)$, and
\[
\underline{\dim}_{loc}\mu_\p(x)=\underline{\dim}_{loc}\mu_\p(y),\quad \overline{\dim}_{loc}\mu_\p(x)=\overline{\dim}_{loc}\mu_\p(y).
\]
\end{proposition}
\begin{proof}
Since the proofs for different cases are similar, we assume without loss of generality that $f_1(I)\cap f_2(I)\ne\emptyset$ and $f_{m-1}(I)\cap f_m(I)\ne\emptyset$, where $I=conv(E)$ is the convex hull of $E$.
Let
\[S=\bigcup_{i=1}^{m-1}\big(f_i(E)\cap f_{i+1}(E)\big).\]
 Then any $x\in S$ has at least two different codings. Define the expanding map on $E$ by
\[
T: E\rightarrow E;\quad x\mapsto T(x)=\left\{
\begin{array}{lll}
f_i^{-1}(x)&\textrm{if}& x\in f_i(E)\setminus S\\
f_i^{-1}(x)&\textrm{if}& x\in f_i(E)\cap f_{i+1}(E).
\end{array}\right.
\]
Then $\uu_1(\Phi)=\set{x\in E: T^n(x)\notin S~\forall n\ge 0}$.

Fix an integer $k\ge 2$ and take  $x\in\uu_k(\Phi)$. Then there exists a smallest integer $k_1\ge 0$ such that $y_0:=T^{k_1}(x)\in S$. So there exists a unique block $i_1\ldots i_{k_1}\in\set{1,\ldots, m}^{k_1}$ (it is the empty block $\epsilon$ if $k_1=0$) such that
$
x=f_{i_1\ldots i_{k_1}}(y_0).
$
Since $y_0\in S$, there exists $j_1\in\set{1,\ldots, m-1}$ such that
\[
y_0\in f_{j_1}(E)\cap f_{j_1+1}(E)=f_{j_1m^u}(E)=f_{(j_1+1)1^v}(E),
\]
where $u=u_{j_1}, v=v_{j_1}\in\N$ are the overlapping indices. Note that $x\notin\uu_{\aleph_0}(\Phi)$. By Corollary \ref{cor:countable} it follows that $y_0\notin\set{f_{j_1m^u}(a), f_{(j_1+1)1^v}(b)}$. So, either there exists an integer $\ell_1\ge 0$ such that
\begin{equation}\label{eq:finite-1}
y_0\in f_{j_1 m^u 1^{\ell_1}}(E)\setminus f_{j_1 m^u 1^{\ell_1+1}}(E),
\end{equation}
or there exists an integer $\ell_1'\ge 0$ such that
\begin{equation}\label{eq:finite-2}
y_0\in f_{(j_1+1)1^{v} m^{\ell_1'}}(E)\setminus f_{(j_1+1)1^v m^{\ell_1'+1}}(E).
\end{equation}

Since the proof for the case in (\ref{eq:finite-2}) can be handled similarly, without loss of generality we assume (\ref{eq:finite-1}) holds.
Note that  $f_{m-1}(I)\cap f_m(I)=f_{(m-1)m^p}(I)=f_{m1^q}(I)$ for some positive integers $p=u_{m-1}, q=v_{m-1}$. Then  by using $f_{jm^u=f_{(j+1)1^v}}$ and $f_{m1^q}=f_{(m-1)m^p}$ it follows that
\[
f_{(j_1+1)1^{v+\ell_1}}=f_{j_1 m^u 1^{\ell_1}}=f_{j_1m^{u-1}(m-1)m^p 1^{\ell_1-q}}=\cdots= f_{j_1m^{u-1}((m-1)m^{p-1})^s m1^{\ell_1-s q}}\
\]
for all $s=0,1,\ldots, \lfloor\frac{\ell_1}{q}\rfloor.$
Here $\lfloor r\rfloor$ stands for the integer part of a real number $r$.
So, by (\ref{eq:finite-1}) there exists a unique $y_1\in E$ and $N_1:=\lfloor\frac{\ell_1}{q}\rfloor+2$ different blocks
\begin{align*}
&W_{1,1}:=(j_1+1)1^{v+\ell_1},\\
& W_{1,2}:=j_1 m^u1^{\ell_1},\\
& W_{1,3}:=j_1m^{u-1}(m-1)m^{p}1^{\ell_1-q},\\
&\quad\vdots\\
& W_{1,s+2}:=j_1m^{u-1}((m-1)m^{p-1})^s m1^{\ell_1-sq},\\
&\quad \vdots\\
&W_{1,N_1}:=j_1m^{u-1}((m-1)m^{p-1})^{N_1-2} m 1^{l_1-(N_1-2)q},
\end{align*}
such that
\[y_0=f_{W_{1,i}}(y_1)\quad\textrm{for all }1\le i\le N_1.\]
 This implies that $y_0$ has $N_1$ different codings landing on $y_1$. Furthermore, note that $f_1(I)\cap f_m(I)=\emptyset$. One can verify that $y_0$ has precisely $N_1$ different codings landing on $y_1$. Observe that $x=f_{i_1\ldots i_{k_1}}(y_0)$ has a unique coding landing on $y_0$. Therefore, $x$ has precisely $N_1$ different codings landing on $y_1$.

If $y_1\in\uu_1(\Phi)$, then $k=N_1$ and the proof is complete by taking $y=y_1$. Otherwise, there exists a smallest integer $k_2\ge 0$ such that $T^{k_2}(y_1)\in S$. By the same argument as above to the point $T^{k_2}(y_1)$ we can find a unique point $y_2\in E$ and $N_2$ different blocks $W_{2,1}, W_{2,2},\ldots, W_{2, N_2}$ such that $f_{W_{2,i}}=f_{W_{2,j}}$ for any $i\ne j$, and $y_1=f_{W_{2,i}}(y_2)$. In other words, $y_1$ has precisely $N_2$ different codings landing on $y_2$. This, combined with the discussion from $x$ to $y_1$, implies that  $x$ has precisely $M_2\le N_1\cdot N_2$ different codings landing on $y_2$. We emphasize that $M_2$ might be strictly smaller than $N_1\cdot N_2$ if $y_1\in S$.

Since $k$ is finite, proceeding the above arguments for finitely many times we can find a unique point $y_J\in\uu_1(\Phi)$ and $M_J$ different blocks $W_{J,1}, W_{J,2}, \ldots, W_{J, M_J}$ such that
\begin{equation}\label{eq:finite-3}
x=f_{W_{J, 1}}(y_J),\quad\textrm{and}\quad f_{W_{J, i}}=f_{W_{J, j}}\quad\forall ~i\ne j.
\end{equation}
Furthermore, $x$ has precisely $M_J$ different codings landing on $y_J$. Since $y_J\in\uu_1(\Phi)$, this implies that $k=M_J$.

  In the following it suffices to prove that the local dimension of $\mu_\p$ at $x$ is the same as that at $y_J$. Note that the contraction ratios of $f_{W_{J, i}}, 1\le i\le M_J$ are the same, denote it by $r_J$. Define
\[
C_J:=\set{i_1\ldots i_n\in\set{1,\ldots,m}^*: \prod_{\ell=1}^n r_{i_\ell}\le r_J<\prod_{\ell=1}^{n-1}r_{i_\ell} }.
\]
Then $W_{J, i}\in C_J$ for all $1\le i\le M_J$, and
\[\set{1,\ldots, m}^\N=\bigcup_{\mathbf i\in C_J}[\mathbf i],\] where the union  is pairwise disjoint. Here $[\mathbf i]$ is a cylinder set generated by the block $\mathbf i$. So, for $r>0$
\begin{equation}\label{eq:measure-ball}
\begin{split}
\mu_\p(B(x, r))&=\sum_{W\in C_J}p_W\mu_\p\circ f_W^{-1}(B(x, r))\\
&=\sum_{i=1}^{M_J} p_{W_{J,i}}\mu_\p\circ f_{W_{J, i}}^{-1}(B(x, r))+\sum_{W\in C_J\setminus\set{W_{J, i}}_{i=1}^{M_J}}p_W\mu_\p\circ f_{W}^{-1}(B(x, r)).
\end{split}
\end{equation}
Observe that for any $W\in C_J\setminus\set{W_{J, i}}_{i=1}^{M_J}$ and for sufficiently small $r>0$ the open set $B(x, r)$ is separated from $f_W(E)$. So the second summation in (\ref{eq:measure-ball}) will disappear for small $r>0$. By (\ref{eq:finite-3}), (\ref{eq:measure-ball}) and Lemma \ref{lem:scaling property} it follows that for sufficiently small $r>0$,
\begin{align*}
\mu_\p(B(x, r))&=\sum_{i=1}^{M_J} p_{W_{J, i}}\;\mu_\p\Big(B(f_{W_{J, i}}^{-1}(x), r_{W_{J,i}}^{-1}r)\Big)=\mu_\p(B(y_J, r_J^{-1}r))\sum_{i=1}^{M_J}p_{W_{J,i}}.
\end{align*}
Observe that $\sum_{i=1}^{M_J}p_{W_{J, i}}$ and $r_J^{-1}$ are both  positive constants independent of $r$.
This implies
\[\underline{\dim}_{loc}\mu_\p(x)=\underline{\dim}_{loc}\mu_\p(y_J)\quad \textrm{and}\quad \overline{\dim}_{loc}\mu_\p(x)=\overline{\dim}_{loc}\mu_\p(y_J).\]
\end{proof}

By Proposition \ref{prop:finite-unique} the local dimension of $\mu_\p$ at points in $\uu_k(\Phi)$ is uniquely  determined by the local dimension of $\mu_\p$ at points in $\uu_1(\Phi)$. In the following result we explicitly determine the local dimension of $\mu_\p$ at points in $\uu_1(\Phi)$.
\begin{proposition}
\label{prop:local-unique}
Let $\Phi=(E, \set{f_i}_{i=1}^m)\in\ee$ with $f_i(x)=r_i x+b_i$ for $1\le i\le m$. Then for any $y\in\uu_1(\Phi)$ with its unique coding $(j_k)=j_1j_2\ldots\in\set{1,\ldots, m}^\N$ we have
\[
\underline{\dim}_{loc}\mu_\p(y)=\liminf_{n\to\f}\frac{\sum_{k=1}^n\log p_{j_k}}{\sum_{k=1}^n\log r_{j_k}},\quad \overline{\dim}_{loc}\mu_\p(y)=\limsup_{n\to\f}\frac{\sum_{k=1}^n\log p_{j_k}}{\sum_{k=1}^n\log r_{j_k}}.
\]
\end{proposition}
\begin{proof}
Let $y\in\uu_1(\Phi)$ with its unique coding $(j_k)$. Note by Condition (C) that  there exists at least one pair of disjoint neighboring basic intervals, i.e.,  $f_i(I)\cap f_{i+1}(I)=\emptyset$ for some $1\le i<m$.    Let $I=[a, b]$ be the convex hull of $E$, and let
\[
g:=\min\set{f_{i+1}(a)-f_i(b): ~f_i(I)\cap f_{i+1}(I)=\emptyset}.
\]
Then $g>0$ is the length of the smallest gap between the neighboring basic intervals. Denote by $r_{max}:=\max_{1\le i\le m}r_i$. Then there exists a large integer $N_1\ge 1$ such that
\begin{equation}\label{eq:april-28-1}
r_{max}^{N_1}(b-a)<g.
\end{equation}
Recall from Definition \ref{def:overlap-vector} that $(u_1,\ldots, u_m)$ and $(v_1,\ldots, v_m)$ are the overlapping vectors. Let
\[
N_2:=\max\set{\max_{1\le i\le m, u_i\ne \f} u_i, ~\max_{1\le i\le m, v_i\ne\f}v_i}.
\]
Take $n>N_2$ sufficiently large, and let
\[R_n:=r_{max}^{N_1}(b-a)\prod_{k=1}^n r_{j_k}.\]
 Clearly, the ball $B(y, R_n)$ contains the interval $f_{j_1\ldots j_{n+N_1}}(I)$. On the other hand, by (\ref{eq:april-28-1}) we have $R_n<g \prod_{k=1}^n r_{j_k}$. Observe that the basic intervals of higher level have similar geometrical structure as in the first level. So, $B(y, R_n)\cap I$ is included in the basic inerval $f_{j_1\ldots j_{n-N_2}}(I)$.

 Therefore,
\[
f_{j_1\ldots j_{n+N_1}}(I)\subset B(y, R_n)\cap I\subset f_{j_1\ldots j_{n-N_2}}(I),
\]
which implies
\begin{equation}\label{eq:april-28-2}
\prod_{k=1}^{n+N_1}p_{j_k}\le \mu_\p(B(y, R_n))\le \prod_{k=1}^{n-N_2}p_{j_k}.
\end{equation}
Here we emphasize that the integers $N_1, N_2$ are independent of $n$.
Taking the logarithms and dividing $\log R_n$ on both sides of Equation (\ref{eq:april-28-2}) yields
\begin{equation}\label{eq:april-28-3}
\frac{\sum_{k=1}^{n-N_2}\log p_{j_k}}{\log R_n}\le \frac{\log\mu_\p(B(y, R_n))}{\log R_n}\le \frac{\sum_{k=1}^{n+N_1}\log p_{j_k}}{\log R_n}.
\end{equation}
Denote by $p_{min}:=\min_{1\le i\le m}p_i$ and $r_{min}:=\min_{1\le i\le m}r_i$. Using the inequalities $p_i\ge p_{min}$ and $r_{min}\le r_i\le r_{max}$ in (\ref{eq:april-28-3}) it follows that
\begin{align*}
\frac{\log \mu_\p(B(y, R_n))}{\log R_n}&\ge \frac{\sum_{k=1}^{n-N_2}\log p_{j_k}}{\sum_{k=1}^{n-N_2}\log r_{j_k}+N_2\log r_{min}+N_1\log r_{max}+\log (b-a)}\\
\frac{\log \mu_\p(B(y, R_n))}{\log R_n}&\le \frac{\sum_{k=1}^{n-N_2}\log p_{j_k}+(N_1+N_2)\log p_{min}}{\sum_{k=1}^{n-N_2}\log r_{j_k}+(N_1+N_2)\log r_{max}+\log(b-a)}.
\end{align*}
Letting $n\to\f$ we obtain
\begin{equation}\label{eq:april-28-4}
\begin{split}
\liminf_{n\to\f}\frac{\log \mu_\p(B(y, R_n))}{\log R_n}&=\liminf_{n\to\f}\frac{\sum_{k=1}^n\log p_{j_k}}{\sum_{k=1}^n\log r_{j_k}},\\
\limsup_{n\to\f}\frac{\log \mu_\p(B(y, R_n))}{\log R_n}&=\limsup_{n\to\f}\frac{\sum_{k=1}^n\log p_{j_k}}{\sum_{k=1}^n\log r_{j_k}}.
\end{split}
\end{equation}

Now for $R_{n+1}<r\le R_n$ we have
\begin{equation}\label{eq:april-28-5}
\frac{\log \mu_\p(B(y, R_n))}{\log R_{n+1}}\le\frac{\log \mu_\p(B(y, r))}{\log r}\le \frac{\log\mu_\p(B(y, R_{n+1}))}{\log R_n}.
\end{equation}
Since $r_{min}\le \frac{R_{n+1}}{R_n}\le r_{max}$ and $R_n\to 0$ as $n\to\f$, we have $\frac{\log R_{n+1}}{\log R_n}\to 1$ as $n\to\f$. By (\ref{eq:april-28-4}) and (\ref{eq:april-28-5}) we conclude that
\[
\underline{\dim}_{loc}\mu_\p(y)=\liminf_{r\to 0}\frac{\log\mu_\p(B(y,r))}{\log r}=\liminf_{n\to\f}\frac{\log \mu_\p(B(y, R_n))}{\log R_n}=\liminf_{n\to\f}\frac{\sum_{k=1}^n\log p_{j_k}}{\sum_{k=1}^n\log r_{j_k}},
\]
and similarly,
\[\overline{\dim}_{loc}\mu_\p(y)=\limsup_{n\to\f}\frac{\sum_{k=1}^n\log p_{j_k}}{\sum_{k=1}^n\log r_{j_k}}.\]
\end{proof}

\subsection{Local dimension of $\mu_\p$ at points in $\uu_{\aleph_0}(\Phi)$}
Recall from Proposition \ref{cor:countable} that any $x\in\uu_{\aleph_0}(\Phi)$ must be of the form
\[
x=f_{\mathbf i}(f_1(b))\quad\textrm{if}\quad f_1(I)\cap f_2(I)\ne\emptyset
\]
for some $\mathbf i\in\set{1,\ldots, m}^*$, or of the form
\[
x=f_{\mathbf j}(f_m(a))\quad\textrm{if}\quad f_{m-1}(I)\cap f_m(I)\ne\emptyset
\]
for some $\mathbf j\in\set{1,\ldots, m}^*$.

The following lemma can be shown by a similar way as in the proof of Proposition \ref{prop:finite-unique}.
\begin{lemma}\label{lem:countable-finite}
Let $\Phi=(E, \set{f_i}_{i=1}^m)\in\ee$ with the convex hull $conv(E)=[a, b]$.
\begin{itemize}
\item If $x=f_{\mathbf i}(f_1(b))\in\uu_{\aleph_0}(\Phi)$ with $\mathbf i\in\set{1,\ldots, m}^*$, then
\[\dim_{loc}\mu_\p(x)=\dim_{loc}\mu_\p(f_1(b)).\]

\item If $x=f_{\mathbf j}(f_m(b))\in\uu_{\aleph_0}(\Phi)$ with $\mathbf j\in\set{1,\ldots, m}^*$, then
\[\dim_{loc}\mu_\p(x)=\dim_{loc}\mu_\p(f_m(b)).\]
\end{itemize}
\end{lemma}

By Lemma \ref{lem:countable-finite} it suffices to consider the local dimension of $\mu_\p$ at   $f_1(b)$ and $f_m(a)$.
\begin{proposition}
\label{prop:local-dim-countable}
Let $\Phi=(E, \set{f_i}_{i=1}^m)\in\ee$ with $I=conv(E)=[a, b]$, and let $\mathbf u=(u_1,\ldots, u_m), \mathbf v=(v_1,\ldots, v_m)$ be the overlapping vector.
\begin{enumerate}[{\rm(i)}]
\item If $f_1(I)\cap f_2(I)\ne\emptyset$, then
\begin{align*}
\dim_{loc}\mu_\p(f_1(b))&=\min\set{\frac{\log p_m}{\log r_m}, \frac{\log p_2+(v_1-1)\log p_1}{u_1\log r_m}}.
\end{align*}
\item If $f_{m-1}(I)\cap f_m(I)\ne\emptyset$, then
\begin{align*}
\dim_{loc}\mu_\p(f_m(a))&=\min\set{\frac{\log p_1}{\log r_1}, \frac{\log p_{m-1}+(u_{m-1}-1)\log p_m}{v_{m-1}\log r_1}}.
\end{align*}
\end{enumerate}
\end{proposition}
\begin{proof}
Since the proof of (ii) is similar, we only prove (i).  Suppose $f_1(I)\cap f_2(I)\ne\emptyset$. Then there exist integers  $u_1, v_1\ge 1$ such that  $f_1(I)\cap f_2(I)=f_{1m^{u_1}}(I)=f_{21^{v_1}}(I)$. We will determine the local dimension of $\mu_\p$ at  $x=f_1(b)\in\uu_{\aleph_0}(\Phi)$.

Denote by $\rho:=\min\set{b-f_m(a), b-f_{m-1}(b)}$.  Then $\rho>0$, and the ball $B(b, \rho)$ has empty intersection with $f_i(I)$ for any $1\le i<m$. This implies $B(b, \rho)\subset f_m(I)$.  On the other hand, since the set sequence $f_{m^n}(I)$ decreases to  $\set{b}$ as $n\to\f$, there exists a large integer $N$ such that $B(b, \rho)\supset f_{m^N}(I)$. Therefore,
\[
f_{m^N}(I)\subset B(b, \rho)\subset f_m(I).
\]
Take a large integer $n$ such that  $R_n:=r_1r_m^n \rho\in(0, 1)$. Note that $x=f_1(b)=f_{1m^n}(b)$. Then by Lemma \ref{lem:scaling property} and the above inclusions it follows that
\begin{equation}\label{eq:april-29-1}
f_{1m^{n+N}}(I)\subset f_{1m^n}(B(b, \rho))=B(x, R_n)\subset f_{1m^{n+1}}(I).
\end{equation}
Observe that $f_{1m^{u_1}}=f_{21^{v_1}}$. Then for any $k\in\N_{\ge u_1}$,
\begin{equation}\label{eq:april-29-2}
f_{1m^k}(I)=f_{(21^{v_1-1})^s 1m^{k-s u_1}}(I)\quad\textrm{for all }s=0,1,\ldots, \lfloor\frac{k}{u_1}\rfloor.
\end{equation}
Applying (\ref{eq:april-29-2}) to (\ref{eq:april-29-1}) yields that for $n\ge u_1$
\[
\bigcup_{s=0}^{\lfloor\frac{n+N}{u_1}\rfloor} f_{(21^{v_1-1})^s 1m^{n+N-s u_1}}(I)\subset B(x, R_n)\subset\bigcup_{s=0}^{\lfloor\frac{n+1}{u_1}\rfloor}f_{(21^{v_1-1})^s 1m^{n+1-s u_1}}(I).
\]
This implies
\begin{equation}\label{eq:april-29-3}
\max_{0\le s\le \lfloor\frac{n+N}{u_1}\rfloor} (p_2 p_1^{v_1-1})^s p_1 p_m^{n+N-s u_1}\le \mu_\p(B(x, R_n))\le \sum_{s=0}^{\lfloor\frac{n+1}{u_1}\rfloor}(p_2 p_1^{v_1-1})^s p_1 p_m^{n+1-s u_1}.
\end{equation}
Now we split the proof into the following two cases.

Case I. $p_2 p_1^{v_1-1}\le p_m^{u_1}$. Then taking the logarithms and dividing $\log R_n$ on both sides of (\ref{eq:april-29-3}) it follows that
\begin{align*}
\frac{\log \mu_\p(B(x,R_n))}{\log R_n}&\le \frac{\log\left(\max_{0\le s\le \lfloor\frac{n+N}{u_1}\rfloor}p_1\Big(\frac{p_2p_1^{v_1-1}}{p_m^{u_1}}\Big)^s\right)+(n+N)\log p_m}{n\log r_m+\log(r_1\rho)}\\
\frac{\log \mu_\p(B(x,R_n))}{\log R_n}&\ge \frac{\log\left(\sum_{s=0}^{\lfloor\frac{n+1}{u_1}\rfloor}p_1\Big(\frac{p_2p_1^{v_1-1}}{p_m^{u_1}}\Big)^s\right)+(n+1)\log p_m}{n\log r_m+\log(r_1\rho)}.
\end{align*}
Note that  $p_2 p_1^{v_1-1}\le p_m^{u_1}$. Letting $n\to\f$ in the above equation gives
\begin{equation}\label{eq:april-29-4}
\lim_{n\to\f}\frac{\log \mu_\p(B(x, R_n))}{\log R_n}=\frac{\log p_m}{\log r_m}.
\end{equation}

Case II. $p_2 p_1^{v_1-1}> p_m^{u_1}$. Then the inequalities in (\ref{eq:april-29-3}) can be rearranged as
\begin{align*}
 \mu_\p(B(x, R_n))&\ge (p_2 p_1^{v_1-1})^{\frac{n+N}{u_1}} \max_{0\le s\le \lfloor\frac{n+N}{u_1}\rfloor}   p_1\Big(\frac{p_m^{u_1}}{p_2p_1^{v_1-1}}\Big)^{\frac{n+N-su_1}{u_1}},\\
 \mu_\p(B(x, R_n))&\le (p_2 p_1^{v_1-1})^{\frac{n+1}{u_1}}\sum_{s=0}^{\lfloor\frac{n+1}{u_1}\rfloor}p_1  \Big(\frac{p_m^{u_1}}{p_2p_1^{v_1-1}}\Big)^{\frac{n+1-su_1}{u_1}}.
\end{align*}
By using $p_2 p_1^{v_1-1}> p_m^{u_1}$ and a similar argument as in Case I one can verify that
\begin{equation}\label{eq:april-29-5}
\lim_{n\to\f}\frac{\log\mu_\p(B(x, R_n))}{\log R_n}=\frac{\log p_2+(v_1-1)\log p_1}{u_1\log r_m}.
\end{equation}

By (\ref{eq:april-29-4}) and (\ref{eq:april-29-5}) it follows that
\[
\lim_{n\to\f}\frac{\log\mu_\p(B(x, R_n))}{\log R_n}=\min\set{\frac{\log p_m}{\log r_m}, \frac{\log p_2+(v_1-1)\log p_1}{u_1\log r_m}}.
\]
Note that $R_{n+1}/R_n$ is bounded away from zero and infinity, and $R_n\to 0$ as $n\to\f$. Therefore, we can conclude from the above equation that
\[
\dim_{loc}\mu_\p(x)=\lim_{r\to 0}\frac{\log \mu_\p(B(x, r))}{\log r}=\min\set{\frac{\log p_m}{\log r_m}, \frac{\log p_2+(v_1-1)\log p_1}{u_1\log r_m}}.
\]
\end{proof}
\begin{proof}
[Proof of Theorem \ref{th:local dimensions}]
For the local dimension of $\mu_\p$ at points in $\uu_k(\Phi)$ it follows from Propositions \ref{prop:finite-unique} and \ref{prop:local-unique}. And for the local dimension of $\mu_\p$ at points in $\uu_{\aleph_0}(\Phi)$ it can be deduced from Lemma \ref{lem:countable-finite} and Proposition \ref{prop:local-dim-countable}.
\end{proof}

\section{Final remarks}\label{sec:final remarks}
We believe some of the results obtained in this paper can be extended to a much more general class of SIFS (see e.g., \cite{Dajani_Jiang_2015}).   Observe that  Condition (D) in our class $\ee$, also called the complete overlap condition, is very strong. For a possible extension one might think of dropping out this complete overlap condition.
\begin{example}
Let $E$ be the attractor of the IFS $\{f_i(x)\}_{i=1}^{m}$ with $m\geq 5$. Denote the convex hull of $E$ by $I=[a,b]$. Suppose $\Phi=(E, \set{f_i}_{i=1}^m)$ satisfies the following conditions.
 \begin{itemize}
 \item $f_{1m}=f_{21}$;
 \item $f_{2}(I)\cap f_{i}(I)=\emptyset$ for any $i>2$;
 \item $f_{m}\cap f_j(I)=\emptyset$ for any   $j<m$;
  \item  $f_{i}(I)\subset (f_2(b), f_{m}(a))$ for any   $3\leq i\leq m-1$.
 \end{itemize}
Observe by the last condition that we have a lot of flexibility for the locations of $f_i(I)$ for $3\le i\le m-1$. Then $\Phi$ does not necessarily belong  to $\ee$. But one can still show that $\dim_H\uu_k(\Phi)=\dim_H\uu_1(\Phi)$ for any $k\in\N$.
\end{example}

The object studied in this paper is in one dimension. It would be interesting to consider a   higher dimensional analogue.

\section*{Acknowledgements}
 The authors wish to thank Dr.~Junjie Miao for some discussion on the earlier version of Theorem \ref{th:local dimensions}. Jiang was supported by NSFC No.~11701302 and K.C. Wong Magna Fund in Ningbo University. Kong was supported by NSFC No.~11401516, and by the Fundamental Research Funds for the Central Universities No.~2019CDXYST0015. Li was  supported by NSFC No.~11671147, 11571144 and Science and Technology Commission of Shanghai Municipality (STCSM)  No.~18dz2271000. Xi was supported by  NSFC No.~11831007.


\end{document}